\documentclass[10pt,a4paper]{article}

\usepackage{amsmath,amsfonts,amsthm,amssymb,amscd, appendix, verbatim}
\usepackage{xpatch}
\usepackage[utf8]{inputenc}
\usepackage{setspace, color}
\usepackage{array}
\usepackage{cite}

\usepackage{mathtools}

\usepackage[colorlinks=true, linkcolor=blue, citecolor=blue, urlcolor=blue]{hyperref}

\newcommand{\caixa}{\hglue15.7cm$\square$\vspace{5mm}}
\newtheorem{theorem}{Theorem}[section]
\newtheorem{corollary}[theorem]{Corollary}

\newtheorem{lemma}[theorem]{Lemma}

\theoremstyle{definition}

\newtheorem{remark}[theorem]{Remark}

\newcommand{\R}{\mathbb{R}}

\makeatletter
   \xpatchcmd{\@thm}{\fontseries\mddefault\upshape}{}{}{} 
\makeatother

\usepackage[margin=1in]{geometry}

\title{Local existence of solutions and blow-up criteria for the Boussinesq equations in Lei-Lin-Gevrey Spaces}
\author{
Patrícia L. Guidolin\footnote{Departamento de Matemática Pura e Aplicada, Universidade Federal do Rio Grande do Sul, Porto
Alegre, RS 91509-900, Brazil, e-mail: patricia.guidolin@ufrgs.br}\,, Wilberclay G. Melo\footnote{Corresponding author: Departamento de Matemática, Universidade Federal de Sergipe, São Cristóvão SE 49100-000,
  Brazil, e-mail: wilberclay@academico.ufs.br}\,, and Thyago S. R. Santos\footnote{Departamento de Matemática,  Instituto de Matemática, Estatística e Computação Científica, Universidade Estadual de Campinas, Campinas SP 13083-859, Brazil e-mail: thyagosr@unicamp.br. This author was partially supported by FAPESP grant No. 2024/15587-1.}}
\date{}

\begin{document}

\maketitle

\begin{abstract}
\noindent This paper studies the local well-posedness and the behavior at potential finite blow-up times of mild solutions to the three-dimensional fractional Boussinesq equations in Lei-Lin and Lei-Lin-Gevrey spaces $\mathcal{X}_{a,\sigma}^s(\mathbb{R}^3)$. By combining fixed point arguments with Fourier estimates adapted to these spaces, we obtain local existence and uniqueness results for data in $\mathcal{X}_{a,\sigma}^s(\mathbb{R}^3)$, including the usual Lei-Lin case $a=0$. We also establish several blow-up criteria for maximal mild solutions and derive lower bounds for the growth of the corresponding norms as the maximal time is approached. In the strict Lei-Lin-Gevrey regime, and when the dissipative exponents coincide, these estimates yield an exponential type blow-up criterion.
\end{abstract}

\textbf{Key words:} {\it Boussinesq equations; local existence of mild solutions; blow-up criteria; Lei-Lin spaces; Gevrey class}

\textbf{AMS Mathematics Subject Classification:} 35B44, 35A01, 35A02, 35Q30, 35Q35.

\section{Introduction}

The goal of this paper is to investigate the local solvability and the behavior at possible finite blow-up times of mild solutions to the three-dimensional fractional Boussinesq equations in Lei-Lin and Lei-Lin-Gevrey spaces. More precisely, we study
\begin{equation}\label{MHD-alpha}
\left\{
\begin{array}{l}
u_t
\;\!+\,
(-\Delta)^{\alpha}\,u
\,+\,
u \cdot \nabla u
\,+\,
\nabla \;\!p \:\!
\;=\;
\theta\, e_3, \quad x\in \mathbb{R}^3, t>0,\\
%
%
%
%
\theta_t
\;\!\,+\,
(-\Delta)^{\beta}\,\theta
\,\,+\,
u \cdot \nabla \theta
\;=\;
0, \quad x\in \mathbb{R}^3, t> 0,\\
%
%
\mbox{div}\:u  \;=\; 0, \quad x\in \mathbb{R}^3, t>0,\\
%
%
u(x,0) \,=\, u_0(x),
\;\,
\theta(x,0) \,=\, \theta_0(x), \quad x\in \mathbb{R}^3,
\end{array}
\right.
\end{equation}
where $u(x,t)=(u_1(x,t),u_2(x,t),u_3(x,t))\in\mathbb{R}^3$ denotes the incompressible velocity field, $p=p(x,t)$ is the scalar pressure, and $\theta=\theta(x,t)\in\mathbb{R}$ represents the temperature, density fluctuation, or buoyancy variable, according to the physical interpretation under consideration. The vector $e_3=(0,0,1)$ is the vertical direction of gravity, and the force $\theta e_3$ is the buoyancy forcing. Throughout the paper we assume
$\alpha,\beta\in[\tfrac12,\infty)$ and $\operatorname{div}u_0=0$. For a suitable function $f$, the fractional Laplacian is defined in the Fourier side by
$$
\mathcal{F}[(-\Delta)^\gamma f](\xi)=|\xi|^{2\gamma}\widehat{f}(\xi).
$$
Thus the parameters $\alpha$ and $\beta$ measure, respectively, the strength of the velocity dissipation and of the thermal diffusion. The borderline value $1/2$ is especially delicate in the present approach, since the dissipation has exactly the strength needed to compensate one spatial derivative coming from the transport nonlinearities.

The Boussinesq system (\ref{MHD-alpha}) is a fundamental model in fluid mechanics and geophysical fluid dynamics. It describes incompressible flows in which density or temperature variations are small enough to be neglected everywhere except in the gravitational force. In this sense, it retains the essential coupling between hydrodynamic transport and buoyancy while avoiding the full compressible dynamics. Such equations appear naturally in the study of atmospheric and oceanic motion, thermal convection, reactive fronts and related stratified flows; see, for instance, \cite{Gill1982,MR2245751,MR2540168,cilon1}. Fractional versions of the model have also received substantial attention. They are useful both as mathematical interpolations between weak and strong dissipative mechanisms and as models for anomalous diffusion or nonlocal dissipation in geophysical regimes. See, among others, \cite{MR2379269,MR3008326,MR3503190,MR3759571,MR3808344,MR4795499,MR4884564,thyago5,thyago8} and the references therein.

Observe that, if $\theta\equiv0$, then \eqref{MHD-alpha} reduces to the usual fractional Navier--Stokes equations
\begin{equation}\label{NS}
\left\{
\begin{array}{l}
u_t
\;\!+\,(-\Delta)^{\alpha} u
\;\!+\,u \cdot \nabla u
\,+\,
\nabla \;\!p
\;=\;
0, \quad x\in \mathbb{R}^3,  t > 0,\\
\mbox{div}\:u  \;=\; 0, \quad x\in \mathbb{R}^3, t > 0,\\
u(x,0) \,=\, u_0(x), \quad x\in \mathbb{R}^3.
\end{array}
\right.
\end{equation}
Consequently, the Boussinesq equations (\ref{MHD-alpha}) may be seen as a Navier--Stokes type dynamics forced by an additional unknown $\theta$, whose own evolution is transported by the velocity field. This apparently simple coupling is one of the main sources of difficulty. In the velocity equation, the temperature acts as a linear forcing term; in the temperature equation, the velocity acts through the transport nonlinearity. Therefore, unlike the pure Navier--Stokes case \eqref{NS}, the fixed point map contains not only quadratic terms but also a genuinely coupled linear contribution.

The functional framework of this paper is the scale of Lei-Lin-Gevrey spaces. For $a\geq0$, $\sigma\geq1$ and $s\in\mathbb{R}$, we consider
$$
\mathcal{X}_{a,\sigma}^s(\mathbb{R}^3):=\left\{f\in S'(\mathbb{R}^3):\widehat f\in L^1_{\operatorname{loc}}(\mathbb{R}^3)\ \hbox{and}\ \|f\|_{\mathcal{X}_{a,\sigma}^s}<\infty\right\},
$$
where
$$
\|f\|_{\mathcal{X}_{a,\sigma}^s}:=
\int_{\mathbb{R}^3}|\xi|^s e^{a|\xi|^{1/\sigma}}|\widehat f(\xi)|\,d\xi.
$$
When $a=0$ this space becomes the usual Lei-Lin space $\mathcal{X}^s(\mathbb{R}^3)$. In particular, $\mathcal{X}^0(\mathbb{R}^3)$ is a Fourier $L^1(\mathbb{R}^3)$ space, closely related to the Wiener algebra
$$
\mathcal{A}(\mathbb{R}^3)=\mathcal{F}L^1(\R^3)
:=
\left\{
f\in S'(\mathbb{R}^3): \widehat f\in L^1(\mathbb{R}^3)
\right\},
$$
while the negative indices $s\in[-1,0]$ allow one to include rougher data by weakening the high-frequency weight. The case $\mathcal{X}^{-1}(\mathbb{R}^3)$ is especially meaningful in the Navier--Stokes theory, since it is invariant under the classical Navier--Stokes scaling and was central in the work of Lei and Lin \cite{LeiLin2011}; see also \cite{Nati,toro,wilberclay26} for related developments.

The exponential factor $e^{a|\xi|^{1/\sigma}}$ gives the Gevrey refinement of the Lei-Lin scale. Intuitively, membership in $\mathcal{X}_{a,\sigma}^s(\mathbb{R}^3)$ says that the Fourier transform decays fast enough at high frequencies after being multiplied by both a polynomial weight $|\xi|^s$ and an exponential Gevrey weight. The parameter $s$ measures Sobolev-type differentiability in an $L^1(\mathbb{R}^3)$ Fourier sense. The parameter $a$ measures the amount of exponential frequency localization and $\sigma$ describes the Gevrey order. When $\sigma=1$, such exponential weights are connected with analytic regularity and complex strip extensions through Paley--Wiener type principles, see \cite{HormanderALPDO1}. For $\sigma>1$, the space corresponds to a Gevrey, generally non-analytic, regime: the high frequencies still decay subexponentially, and this decay encodes a controlled growth of derivatives. This perspective is consistent with the broader role of Gevrey classes in fluid equations, as in \cite{MR1026858,MR2265624,MR2169876,MR3504420,MR4369830,Be14,Be16}.

There are two reasons why the Lei-Lin-Gevrey framework is particularly well adapted to \eqref{MHD-alpha}. First, the fractional heat semigroup $e^{-t(-\Delta)^\gamma}$ has a simple Fourier multiplier representation, and its smoothing effect can be measured directly by moving powers of $|\xi|$ between the dissipative factor and the Lei-Lin-Gevrey norm. Second, the nonlinear terms are naturally estimated by convolution in Fourier variables. The inequality
$$
e^{a|\xi|^{1/\sigma}}\leq e^{a|\xi-\eta|^{1/\sigma}}e^{a|\eta|^{1/\sigma}},\qquad \sigma\geq1,
$$
together with suitable polynomial decompositions of $|\xi|^s$, allows one to control products such as $u\cdot\nabla u$ and $u\cdot\nabla\theta$ in the same scale. In this way the Gevrey weight is not merely an extra regularity assumption, it is a structural device that remains compatible with the bilinear estimates required by the equations.

\begin{remark}
It is useful to distinguish the Lei-Lin-Gevrey framework from the more classical Sobolev-Gevrey setting
$$
\dot{H}_{a,\sigma}^s(\mathbb{R}^3)
    = \left\{ f \in \mathcal{S}'(\mathbb{R}^3) : \widehat{f} \in L^1_{\operatorname{loc}}(\mathbb{R}^3)\,\, \text{and} \,\, \ 
       \|f\|_{\dot{H}_{a,\sigma}^s}:=\Big[\int_{\mathbb{R}^3} |\xi|^{2s} e^{2a |\xi|^{1/\sigma}} |\widehat{f}(\xi)|^2\, d\xi\Big]^{\frac{1}{2}} < \infty
      \right\}.
$$
Sobolev-Gevrey spaces are usually based on weighted Fourier $L^2(\mathbb{R}^3)$ norms and therefore inherit a Hilbert space structure well adapted to energy methods, Plancherel's identity and cancellations in $L^2(\mathbb{R}^3)$-based estimates. In contrast,  Lei-Lin-Gevrey spaces $\mathcal{X}_{a,\sigma}^s(\mathbb{R}^3)$ are built on weighted Fourier $L^1(\mathbb{R}^3)$ norms. This gives a Wiener-type structure in which products are handled through convolution estimates and the fractional heat semigroup acts directly on the norm. This distinction is important for our purposes. In the Sobolev-Gevrey setting, one usually works with energy-type estimates based on $L^2(\mathbb{R}^3)$ norms. In the Lei-Lin-Gevrey setting, the norm is an $L^1(\mathbb{R}^3)$ norm of the Fourier transform, so the heat semigroup and the nonlinear products can be estimated more directly in frequency space. This is useful here because both the construction of mild solutions and the blow-up criteria rely on controlling the same Fourier-based norms.
\end{remark}

After applying  Leray projector $\mathbb{P}$ to remove the pressure, the system \eqref{MHD-alpha} is represented, in the mild form, by
\begin{align}\label{Mild SOlution Form}
\begin{cases}
u(t)=e^{- t(-\Delta)^{\alpha}}u_0 - \displaystyle \int_{0}^te^{- (t-\tau)(-\Delta)^{\alpha}} \mathbb{P}(u\cdot\nabla u)\,d\tau + \int_{0}^te^{- (t-\tau)(-\Delta)^{\alpha}} \mathbb{P}(\theta e_3)\,d\tau,\\
 \theta(t)=e^{- t(-\Delta)^{\beta}}\theta_0 - \displaystyle \int_{0}^te^{-  (t-\tau)(-\Delta)^{\beta}} [u\cdot \nabla\theta]\,d\tau.
\end{cases}
\end{align}
These formulas clarify the strategy of the paper. The local theory is obtained by constructing a fixed point in a space that simultaneously controls $C_T(\mathcal{X}_{a,\sigma}^s(\mathbb{R}^3))$ and the dissipative integrability spaces $L_T^1(\mathcal{X}_{a,\sigma}^{s+2\alpha}(\mathbb{R}^3))$ and $L_T^1(\mathcal{X}_{a,\sigma}^{s+2\beta}(\mathbb{R}^3))$. The blow-up criteria are then derived by combining the local theory with continuation arguments and nonlinear estimates that quantify how the relevant Lei-Lin or Lei-Lin-Gevrey norms must degenerate if the maximal time of existence is finite.

\bigskip

\noindent\textbf{Local existence of mild solutions:} By using an argument of the fixed point, we shall prove the existence of a unique local mild solution for the fractional Boussinesq equations (\ref{MHD-alpha}) in Lei-Lin-Gevrey spaces, by including the usual Lei-Lin spaces as the case $a=0$. Due to the mathematical nature of this system and, more precisely, because of the term $\theta e_3$ in (\ref{MHD-alpha}), it is necessary to use a fixed point result slightly more general than the one usually applied to the Navier--Stokes equations (\ref{NS}) (see Lemma \ref{lemaB} below and Lemma 3.2 in \cite{Nati}). It is relevant to point out that Lemma \ref{lemaB}, with $L=0$, has been used in the study of mild solutions and their behavior for the Navier--Stokes case and related systems in the recent literature, see, for example, \cite{wilberclay26,Nati,artigowilberthyagomanasses,toro,Nata,coriolis,Be14,Be16} and references therein. The novelty here is that the Boussinesq coupling requires the fixed point argument to incorporate the linear forcing produced by the temperature while keeping the sharp Fourier estimates needed for the Lei-Lin-Gevrey scale.

Our first existence result is stated as follows.

\begin{theorem}\label{teoremaexistenciaB2}
Assume that $a\geq0$, $\sigma\geq1$, $\alpha \geq \frac{1}{2}$, $\beta\geq \frac{1}{2}$ and denote $q_{\alpha,\beta}:=\max\big\{1-2\alpha,\tfrac{\alpha(1-2\beta)}{\beta},\tfrac{\beta(1-2\alpha)}{\alpha}\big\}$. Let  $s\in[\max\{-1,q_{\alpha,\beta}\},0]$ and  $(u_0,\theta_0)\in \mathcal{X}_{a,\sigma}^s(\mathbb{R}^3)$. Then, there exist an instant $T>0$ and a unique solution 
$$(u,\theta)\in [C_{T}(\mathcal{X}_{a,\sigma}^s(\mathbb{R}^3))\cap L^1_{T}(\mathcal{X}_{a,\sigma}^{s+2\alpha}(\mathbb{R}^3))]\times [C_{T}(\mathcal{X}_{a,\sigma}^s(\mathbb{R}^3))\cap L^1_{T}(\mathcal{X}_{a,\sigma}^{s+2\beta}(\mathbb{R}^3))]$$ of the Boussinesq equations \emph{(\ref{MHD-alpha})} such that
$$\|u\|_{L^\infty_T(\mathcal{X}_{a,\sigma}^{s})}+\|u\|_{L^1_T(\mathcal{X}_{a,\sigma}^{s+2\alpha})}+\|\theta\|_{L^\infty_T(\mathcal{X}_{a,\sigma}^{s})}+\|\theta\|_{L^1_T(\mathcal{X}_{a,\sigma}^{s+2\beta})}< C_1,$$
where it is considered that $\|(u_0,\theta_0)\|_{\mathcal{X}_{a,\sigma}^{s}}$ is sufficiently small  in the case $s=q_{\alpha,\beta}$. Moreover, it holds 
$$\|u\|_{L^\infty_T(\mathcal{X}_{a,\sigma}^{s})}+\|u\|_{L^1_T(\mathcal{X}_{a,\sigma}^{s+2\alpha})}+\|\theta\|_{L^\infty_T(\mathcal{X}_{a,\sigma}^{s})}+\|\theta\|_{L^1_T(\mathcal{X}_{a,\sigma}^{s+2\beta})}\leq C_2\|(u_0,\theta_0)\|_{\mathcal{X}_{a,\sigma}^{s}}$$
and also,  for every $p\geq 1$, we conclude
$$(u,\theta)\in L^p_{T}(\mathcal{X}_{a,\sigma}^{s+\frac{2\alpha}{p}}(\mathbb{R}^3))\times L^p_{T}(\mathcal{X}_{a,\sigma}^{s+\frac{2\beta}{p}}(\mathbb{R}^3)).$$ Here,  $C_i$ \emph{(}$i=1,2$\emph{)} is a positive constant.
\end{theorem}

In Theorem \ref{teoremaexistenciaB2}, the admissible lower bound for $s$ depends on the dissipative exponents $\alpha$ and $\beta$. This dependence is natural from the viewpoint of the bilinear estimates: the weaker the dissipation is, the closer the regularity index must be to the threshold at which the derivative in the nonlinear transport can be absorbed. In order to obtain a complementary statement in which the range of $s$ is expressed without such explicit dependence on $\alpha$ and $\beta$, we present our second existence result.

\begin{theorem}\label{teoremaexistenciaB}
Assume that $\alpha \geq \frac{1}{2}$, $\beta\geq \frac{1}{2}$,  $(a,s,\sigma)\in\{(0,\infty)\times [-1,0)\times (1,\infty)\}\cup \{[0,\infty)\times\{0\}\times [1,\infty)\}$ and  $(u_0,\theta_0)\in \mathcal{X}_{a,\sigma}^s(\mathbb{R}^3)$. Then, there exist an instant $T>0$ and a unique solution 
$$(u,\theta)\in [C_{T}(\mathcal{X}_{a,\sigma}^s(\mathbb{R}^3))\cap L^1_{T}(\mathcal{X}_{a,\sigma}^{s+2\alpha}(\mathbb{R}^3))]\times [C_{T}(\mathcal{X}_{a,\sigma}^s(\mathbb{R}^3))\cap L^1_{T}(\mathcal{X}_{a,\sigma}^{s+2\beta}(\mathbb{R}^3))]$$ of the Boussinesq equations \emph{(\ref{MHD-alpha})} such that
$$\|u\|_{L^\infty_T(\mathcal{X}_{a,\sigma}^{s})}+\|u\|_{L^1_T(\mathcal{X}_{a,\sigma}^{s+2\alpha})}+\|\theta\|_{L^\infty_T(\mathcal{X}_{a,\sigma}^{s})}+\|\theta\|_{L^1_T(\mathcal{X}_{a,\sigma}^{s+2\beta})}< C_1,$$
where it is considered that $\|(u_0,\theta_0)\|_{\mathcal{X}_{a,\sigma}^{s}}$ is sufficiently small in the case $\alpha=\frac{1}{2}$ or $\beta=\frac{1}{2}$. Moreover, it holds 
$$\|u\|_{L^\infty_T(\mathcal{X}_{a,\sigma}^{s})}+\|u\|_{L^1_T(\mathcal{X}_{a,\sigma}^{s+2\alpha})}+\|\theta\|_{L^\infty_T(\mathcal{X}_{a,\sigma}^{s})}+\|\theta\|_{L^1_T(\mathcal{X}_{a,\sigma}^{s+2\beta})}\leq C_{ 2}\|(u_0,\theta_0)\|_{\mathcal{X}_{a,\sigma}^{s}}$$
and also, for every $p\geq 1$, we conclude
$$(u,\theta)\in L^p_{T}(\mathcal{X}_{a,\sigma}^{s+\frac{2\alpha}{p}}(\mathbb{R}^3))\times L^p_{T}(\mathcal{X}_{a,\sigma}^{s+\frac{2\beta}{p}}(\mathbb{R}^3)).$$ Here, $C_i$ \emph{(}$i=1,2$\emph{)} is a positive constant.
\end{theorem}

\begin{remark}
Let us study the assumptions in Theorems \ref{teoremaexistenciaB2} and \ref{teoremaexistenciaB}. First of all, recall that $\alpha\geq \frac{1}{2}$ and $\beta\geq \frac{1}{2}$ in these results. Thus, it follows that $q_{\alpha,\beta}\leq0$ and consider the two cases below:
\begin{enumerate}
\item Assume that $q_{\alpha,\beta}=0$. In this case, one obtains $\alpha=\frac{1}{2}$ or $\beta=\frac{1}{2}$ and, moreover, we infer $$\max\{-1,1-2\alpha,\tfrac{\alpha(1-2\beta)}{\beta},\tfrac{\beta(1-2\alpha)}{\alpha}\}=\max\{-1,q_{\alpha,\beta}\}=0.$$
Therefore, we need to suppose that $\|(u_0,\theta_0)\|_{\mathcal{X}_{a,\sigma}^{s}}$ is sufficiently small in order to obtain a local mild solution (see Theorems \ref{teoremaexistenciaB2} and \ref{teoremaexistenciaB}), where $$(a,s,\sigma) \in\{(0,\infty)\times [-1,0)\times (1,\infty)\}\cup \{[0,\infty)\times\{0\}\times [1,\infty)\}.$$
\item Consider that $q_{\alpha,\beta}<0$. Thereby, one infers $\alpha>\frac{1}{2}$ and $\beta>\frac{1}{2}$ and, furthermore,  we have  a local mild solution (see Theorems \ref{teoremaexistenciaB2} and \ref{teoremaexistenciaB}), where $$(a,s,\sigma) \in [0,\infty)\times[-1,0]\times [1,\infty)$$ for any initial data $(u_0,\theta_0)\in\mathcal{X}_{a,\sigma}^{s}(\mathbb{R}^3)$ if $q_{\alpha,\beta}<-1$,   $$(a,s,\sigma) \in\{(0,\infty)\times[-1,q_{\alpha,\beta}]\times (1,\infty)\}\cup\{[0,\infty)\times(q_{\alpha,\beta},0]\times [1,\infty)\}$$ for any initial data $(u_0,\theta_0)\in\mathcal{X}_{a,\sigma}^{s}(\mathbb{R}^3)$ if $q_{\alpha,\beta}\geq-1$, and $$(a,s,\sigma)\in[0,\infty)\times\{q_{\alpha,\beta}\}\times [1,\infty)$$ if $\|(u_0,\theta_0)\|_{\mathcal{X}_{a,\sigma}^{s}}$ is small enough and $q_{\alpha,\beta}\geq-1$.
\end{enumerate}
Notice that Theorem \ref{teoremaexistenciaB} extends Theorem 1.1 (ii) obtained in \cite{Nati} from the Navier-Stokes  equations (\ref{NS}) to the Boussinesq system (\ref{MHD-alpha}). Moreover, Theorem \ref{teoremaexistenciaB2} shows how to reach the existence and uniqueness of solutions for these systems  in Lei-Lin spaces $\mathcal{X}^s(\mathbb{R}^3)$ ($a=0$), with $s\neq0$, if $\alpha>\frac{1}{2}$ and $\beta>\frac{1}{2}$. Furthermore, if it is assumed that $\alpha,\beta\in(\frac{1}{2},1]$, one has that Theorem \ref{teoremaexistenciaB2} considers the space  $\mathcal{X}_{a,\sigma}^s(\mathbb{R}^3)$, with $$s=\max\{-1,1-2\alpha,\tfrac{\alpha(1-2\beta)}{\beta},\tfrac{\beta(1-2\alpha)}{\alpha}\}=\max\{1-2\alpha,\tfrac{\alpha(1-2\beta)}{\beta},\tfrac{\beta(1-2\alpha)}{\alpha}\},$$ which is not supposed by Theorem 1.1 in \cite{Nata}. In addition, Theorem \ref{teoremaexistenciaB2} also studies the cases $\alpha,\beta>1 $, $\alpha=\frac{1}{2}$ and $\beta=\frac{1}{2}$, these ones are not presented by \cite{Nata}.

\end{remark}

\bigskip

\noindent\textbf{Blow-up criteria for maximal mild solutions:} According to the assumptions and the proofs of Theorems \ref{teoremaexistenciaB2} and \ref{teoremaexistenciaB}, we establish new blow-up criteria for the local solutions of the Boussinesq equations (\ref{MHD-alpha}) obtained in these two results. Such criteria are continuation obstructions: if the maximal time $T^*$ is finite, then the norms controlling the local theory cannot remain bounded up to $T^*$. The results below quantify this obstruction in the Lei-Lin-Gevrey scale and show, in particular, lower bounds with explicit rates as $t\nearrow T^*$. Thus, it is worth to compare our main results below with the ones obtained by \cite{wilberclay26,Nati,Nata,coriolis,thyago7} and references therein.

Let us start enunciating some blow-up criteria for the local solution presented in Theorem \ref{teoremaexistenciaB2}.

\begin{theorem}\label{teoremaB2}
Assume that $a\geq0$, $\sigma \geq1,$ $\alpha>\frac{1}{2}$, $\beta>\frac{1}{2}$,
$ s> \max\{1-2\alpha,\tfrac{\alpha(1-2\beta)}{\beta},\tfrac{\beta(1-2\alpha)}{\alpha}\}$, $ s\in [-1,0]$ and  $(u_0,\theta_0)\in \mathcal{X}_{a,\sigma}^s(\mathbb{R}^3)$.
Consider that $(u,\theta)\in C([0,T^*),\mathcal{X}_{a,\sigma}^{s}(\mathbb{R}^3))$
is  a maximal solution for the Boussinesq equations \emph{(\ref{MHD-alpha})} obtained in Theorem  \emph{\ref{teoremaexistenciaB2}}. If $T^*<\infty$, then
\begin{enumerate}
	\item[\emph{i)}] $\displaystyle \limsup_{t\nearrow T^*} \|(u,\theta)(t)\|_{\mathcal{X}_{a,\sigma}^{s}}=\infty$;
	\item[\emph{ii)}] $\displaystyle\int_t^{T^*}\big[\|u(\tau)\|_{\mathcal{X}_{a,\sigma}^{s}}^{\frac{2\alpha}{2\alpha+s-1}} + \|u(\tau)\|_{\mathcal{X}_{a,\sigma}^{s}}^{\frac{\beta(s+2\alpha)}{\beta(s+2\alpha)-\alpha}}\|\theta(\tau)\|_{\mathcal{X}_{a,\sigma}^{s}}^{\frac{-s\beta}{\beta(s+2\alpha)-\alpha}}+
\|u(\tau)\|_{\mathcal{X}_{a,\sigma}^{s}}^{\frac{-s\alpha}{\alpha(s+2\beta)-\beta}}\|\theta(\tau)\|_{\mathcal{X}_{a,\sigma}^{s}}^{\frac{\alpha(s+2\beta)}{\alpha(s+2\beta)-\beta}}\big]\;d\tau=\infty$;
	\item[\emph{iii)}] $\|(u,\theta)(t)\|_{\mathcal{X}_{a,\sigma}^{s}}^{\frac{2\alpha}{2\alpha+s-1}} + \|(u,\theta)(t)\|_{\mathcal{X}_{a,\sigma}^{s}}^{\frac{2\alpha\beta}{2\alpha\beta+s\beta-\alpha}}
+\|(u,\theta)(t)\|_{\mathcal{X}_{a,\sigma}^{s}}^{\frac{2\alpha\beta}{2\alpha\beta+s\alpha-\beta}}\geq [e^{C(T^*-t)}-1]^{-1}$,
\end{enumerate}
for all $t\in[0,T^*)$, where $C$ is a positive constant.
\end{theorem}

The blow-up criteria for the local solution obtained in Theorem \ref{teoremaexistenciaB} can be  enunciated as follows.

\begin{theorem}\label{teoremaB1}
Assume that $\alpha>\frac{1}{2}$, $\beta>\frac{1}{2}$,
$(a,s,\sigma)\in\{(0,\infty)\times [-1,0)\times (1,\infty)\}\cup \{[0,\infty)\times\{0\}\times [1,\infty)\}$ and  $(u_0,\theta_0)\in \mathcal{X}_{a,\sigma}^s(\mathbb{R}^3)$.
Consider that $(u,\theta)\in C([0,T^*),\mathcal{X}_{a,\sigma}^{s}(\mathbb{R}^3))$
is  a maximal solution for the Boussinesq equations \emph{(\ref{MHD-alpha})} obtained in Theorem \emph{\ref{teoremaexistenciaB}}. If $T^*<\infty$, then
\begin{enumerate}
	\item[\emph{i)}] $\displaystyle \limsup_{t\nearrow T^*} \|(u,\theta)(t)\|_{\mathcal{X}_{a,\sigma}^{s}}=\infty$;
	\item[\emph{ii)}] $\displaystyle\int_t^{T^*}\big[\|(u,\theta)(\tau)\|_{\mathcal{X}_{a,\sigma}^{s}}^{\frac{2\alpha}{2\alpha-1}} + \|u(\tau)\|_{\mathcal{X}_{a,\sigma}^{s}}^{\frac{2\beta}{2\beta-1}}\big]\;d\tau=\infty$;
	\item[\emph{iii)}] $\|(u,\theta)(t)\|_{\mathcal{X}_{a,\sigma}^{s}}^{\frac{2\alpha}{2\alpha-1}} + \|(u,\theta)(t)\|_{\mathcal{X}_{a,\sigma}^{s}}^{\frac{2\beta}{2\beta-1}}\geq [e^{C(T^*-t)}-1]^{-1}$,
\end{enumerate}
for all $t\in[0,T^*)$, where $C$ is a positive constant.
\end{theorem}

If we assume that $a>0$ and $\sigma>1$ (the strict Lei-Lin-Gevrey spaces) in the subcritical case $\alpha>\frac{1}{2} $ and $\beta>\frac{1}{2}$ of the Boussinesq equations (\ref{MHD-alpha}), we are able to prove an exponential blow-up criterion for the local solution obtained in Theorem \ref{teoremaexistenciaB} (when $\beta=\alpha$). First of all, it is necessary to establish an inductive process in order to present new blow-up criteria related to some usual Lei-Lin spaces. More precisely, it is important to point out that the criteria presented in Corollary \ref{corollaryB1} below are not present in the literature even in the particular case of the Navier--Stokes equations, for instance (see \cite{wilberclay26,Nati,Nata,coriolis} and references included). Thus, we are ready to enunciate the following blow-up criteria for the local solution given in Theorem \ref{teoremaexistenciaB}.

\begin{corollary}\label{corollaryB1}
Assume that $a>0$, $\sigma>1$, $\alpha>\frac{1}{2}$, $\beta>\frac{1}{2}$, $s\in[-1,0]$  and  $(u_0,\theta_0)\in \mathcal{X}_{a,\sigma}^s(\mathbb{R}^3)$.
Consider that $(u,\theta)\in C([0,T^*),\mathcal{X}_{a,\sigma}^{s}(\mathbb{R}^3))$
is  a maximal solution for the Boussinesq equations \emph{(\ref{MHD-alpha})} obtained in Theorem \emph{\ref{teoremaexistenciaB}}. If $T^*<\infty$, then
\begin{enumerate}
	\item[\emph{i)}] $\displaystyle\int_t^{T^*}\big[\|(u,\theta)(\tau)\|_{\mathcal{X}_{\frac{a}{\sigma(\sqrt{\sigma})^{n-1}},\sigma}^{0}}^{\frac{2\alpha}{2\alpha-1}} + \|u(\tau)\|_{\mathcal{X}_{\frac{a}{\sigma(\sqrt{\sigma})^{n-1}},\sigma}^{0}}^{\frac{2\beta}{2\beta-1}}\big]\;d\tau=\infty$;
	\item[\emph{ii)}] $\|(u,\theta)(t)\|_{\mathcal{X}_{\frac{a}{\sigma(\sqrt{\sigma})^{n-1}},\sigma}^{0}}^{\frac{2\alpha}{2\alpha-1}} + \|(u,\theta)(t)\|_{\mathcal{X}_{\frac{a}{\sigma(\sqrt{\sigma})^{n-1}},\sigma}^{0}}^{\frac{2\beta}{2\beta-1}}\geq [e^{C(T^*-t)}-1]^{-1}$,
	\item[\emph{iii)}] $\|(u,\theta)(t)\|_{\mathcal{X}_{\frac{a}{(\sqrt{\sigma})^{n}},\sigma}^{s}}^{\frac{2\alpha}{2\alpha-1}} + \|(u,\theta)(t)\|_{\mathcal{X}_{\frac{a}{(\sqrt{\sigma})^{n}},\sigma}^{s}}^{\frac{2\beta}{2\beta-1}}\gtrsim [e^{C(T^*-t)}-1]^{-1}$;
	\item[\emph{iv)}] $\displaystyle \limsup_{t\nearrow T^*} \|(u,\theta)(t)\|_{\mathcal{X}_{\frac{a}{(\sqrt{\sigma})^{n-1}},\sigma}^{s}}=\infty$;
\item[\emph{v)}] $\|(u,\theta)(t)\|_{\mathcal{X}^0}^{\frac{2\alpha}{2\alpha-1}}+ \|(u,\theta)(t)\|_{\mathcal{X}^0}^{\frac{2\beta}{2\beta-1}}\geq [e^{C(T^*-t)}-1]^{-1}$;
\item [\emph{vi)}] $\|(u,\theta)(t)\|_{\mathcal{X}^s}^{\frac{2\alpha}{2\alpha-1}}+ \|(u,\theta)(t)\|_{\mathcal{X}^s}^{\frac{2\beta}{2\beta-1}}\gtrsim [e^{C(T^*-t)}-1]^{-1}$,
\end{enumerate}
for all $t\in[0,T^*)$ and $n\in \mathbb{N}$, where $C$ denotes positive constants.
\end{corollary}

Still assuming $a>0$ and $\sigma>1$, and supposing $\beta=\alpha>\frac{1}{2}$, our last main result is related to the exponential blow-up criterion for the local solution obtained in Theorem \ref{teoremaexistenciaB}. This result should be read as a quantitative refinement of the preceding criteria: not only must a relevant norm become unbounded when $T^*<\infty$, but in the strict Lei-Lin-Gevrey setting its lower growth contains an exponential component inherited from the Gevrey weight.

\begin{corollary}\label{corollaryB2}
	Assume that $a>0$, $\sigma> 1$, $\beta=\alpha>\frac{1}{2}$, $s\in[-1,0]$ and $(u_0,\theta_0)\in \mathcal{X}_{a,\sigma}^s(\mathbb{R}^3)\cap L^2(\mathbb{R}^3)$. Consider that
$(u,\theta)\in C([0,T^*),\mathcal{X}_{a,\sigma}^{s}(\mathbb{R}^3))$ is  the maximal solution for the Boussinesq equations \emph{(\ref{MHD-alpha})} obtained in Theorem \emph{\ref{teoremaexistenciaB}}. If $T^*<\infty$, then
\begin{align*}
\|(u,\theta)(t)\|_{\mathcal{X}_{a,\sigma}^{s}}&\gtrsim a^{n_0}
		[e^{C_1(T^*-t)}-1]^{-\frac{2\alpha-1}{2\alpha}(\frac{2s}{3}+\frac{2n_0}{\sigma}+1)}
		\|(u_0,\theta_0)\|_{L^2(\mathbb{R}^3)}^{-\frac{2}{3}(s+\frac{n_0}{\sigma})}\\
& \,\,\,\,\,\times \exp\Big\{aC_2[e^{C_3(T^*-t)}-1]^{-\frac{2\alpha-1}{3\alpha\sigma}}\|(u_0,\theta_0)\|_{L^2(\mathbb{R}^3)}^{-\frac{2}{3\sigma}}\Big\},
\end{align*}
for all $t\in[0,T^*)$, where $C_i$ \emph{(}$i=1,2,3$\emph{)} is a positive constant, $\sigma_0$ is the integer part of $-s\sigma$ and   $n_0>\sigma_0$ is a natural number.
\end{corollary}

\begin{remark}
It is worth to point out that this paper establishes the proof of Corollary \ref{corollaryB1} i) and ii), with $n=1$, for $\alpha>\frac{1}{2}$, $\beta>\frac{1}{2}$ and $(a,s,\sigma)\in\{(0,\infty)\times [-1,0)\times (1,\infty)\}\cup \{[0,\infty)\times\{0\}\times [1,\infty)\}$ or $a\geq0$, $\sigma \geq1,$
$ s> \max\{1-2\alpha,\tfrac{\alpha(1-2\beta)}{\beta},\tfrac{\beta(1-2\alpha)}{\alpha}\}$ and $ s\in [-1,0]$ as well. Furthermore, the norms presented in Corollary \ref{corollaryB1} v) and vi) are finite because of Lemma \ref{lemanovo2} below. In addition, it is important to emphasize that Theorem 1.3 in \cite{Nata} is contained in Theorem \ref{teoremaB2} if $\alpha,\beta\leq 1$, by disregarding  the difference in the mathematical  nature of the terms of the Boussinesq and MHD equations;  however, Theorem \ref{teoremaB2} also presents blow-up criteria in the case $\alpha,\beta>1$. On the other hand, Corollary \ref{corollaryB1} is the main novelty of our blow-up criteria (see \cite{wilberclay26,Nati,Nata,coriolis} and references therein). Let us inform that this result is  relevant to obtain our exponential blow-up criterion (see Corollary \ref{corollaryB2} for more details), which is established by \cite{Nata} (for $\beta=\alpha\in(\frac{1}{2},1]$) in the MHD case, as well as this one gives, for instance, a lower bound for our local solution in Lei-Lin spaces $\mathcal{X}^s(\mathbb{R}^3)$, with $s\in[-1,0]$, $\alpha>\frac{1}{2} $ and $\beta>\frac{1}{2}$.
\end{remark}

\bigskip

\noindent\textbf{Organization of the paper:} This paper is organized has follow: Section 2 contains the notation, the definition of Lei-Lin-Gevrey spaces and the auxiliary estimates used throughout the text. Section 3 is devoted to the local existence theory for \eqref{MHD-alpha}, including the fixed point argument and the proof of Theorems \ref{teoremaexistenciaB2} and \ref{teoremaexistenciaB}. Section 4 proves the blow-up criteria for maximal mild solutions presented in Theorems \ref{teoremaB2} and \ref{teoremaB1} and Corollaries \ref{corollaryB1} and \ref{corollaryB2}.

\section{Notations and auxiliary results}

This section is devoted to the presentation of the preliminary notations and results which will be crucial in our arguments (see \cite{pontofixo,Nati,Nata,artigowilberthyagomanasses} for more details).  

\bigskip

\noindent\textbf{Main notations throughout this paper:} As follows, we establish the most relevant notations and definitions that are used throughout our paper.

\begin{enumerate}

\item The functions dependencies are presented in three forms. For example, we denote
$$u=u(t)=u(x,t),\quad x\in \mathbb{R}^3,t\geq0.$$
    
\item $S'(\mathbb{R}^3)$ is the space of tempered distributions.

\item The Fourier transform 
is defined by
 $$\mathcal{F}(f)(\xi)= \widehat{f}(\xi):=\int_{\mathbb{R}^3}e^{-i\xi\cdot x}f(x)\,dx.$$


\item  The tensor product is given by
$$f\otimes g:=(g_{1}f,g_{2}f,g_{3}f),$$
where $f=(f_1,f_2,f_3)$ and $g=(g_1,g_2,g_3)\in S'(\mathbb{R}^3).$

  \item Let $T>0,$  ($X,\|\cdot\|_{X}$)  a normed space and $I\subseteq \mathbb{R}$  an interval. We define
  $$C(I;X)=\{f:I\rightarrow X \hbox{ continuous function}\}, $$
  and  the $C(I;X)$-norm is given by
$$\|f\|_{L^{\infty}(I;X)}:=\sup_{t\in I}\{\|f(t)\|_{X}\}.$$
We denote $C_T(X)=C([0,T];X)$ and $\|\cdot\|_{L_T^{\infty}(X)}=\|\cdot\|_{L^{\infty}([0,T];X)}.$
  \item Let $1\leq p<\infty$, $T>0,$  ($X,\|\cdot\|_{X}$) a normed space and $I\subseteq \mathbb{R}$ an interval. We define
  $$L^p(I;X)=\{f:I\rightarrow X \hbox{ measurable function}: \int_I\|f(t)\|_{X}^p\,dt<\infty\},$$ and the $L^{p}(I;X)$-norm is given by
$$\|f\|_{L^{p}(I;X)}:=\Big(\int_I\|f(t)\|_{X}^p\,dt\Big)^{\frac{1}{p}}.$$
We denote $L_T^{p}(X)=L^{p}([0,T];X)$.

\item The constants in this paper may change their values from line to line without change of notation.
For example, $C_{q}$   denotes any constant that depends on $q$.
\end{enumerate}


\bigskip

\noindent\textbf{Preliminary lemmas:} Below, we present some lemmas that will play an important role in the proofs of our main results. Let us begin with the ones which will be useful in the search for  a mild solution of the Boussinesq equations (\ref{MHD-alpha}) (see \cite{pontofixo,Nati,Nata,artigowilberthyagomanasses} for more information).

\begin{lemma}[see \cite{pontofixo}]\label{lemaB}
		Let $(X,\|\cdot\|)$ be a Banach space, $L:X\rightarrow X$ a continuous linear operator and
 $B:X\times X\rightarrow X$ a continuous bilinear operator, i.e., there exist positive constants $C_1$ and $C_2$ such that
 $$\|L(x)\|\leq C_1 \|x\|,\quad \|B(x,y)\|\leq C_2 \|x\|\|y\|,\quad \forall x,y\in X.$$
 Then, for each $C_1\in (0,1)$ and $x_0\in X$ that satisfy $4C_2\|x_0\| <(1-C_1)^2$, one has that the equation 
 $$x=x_0+B(x,x)+L(x), \quad x\in X,$$
 admits a solution $x\in X$. Moreover, $x$ satisfies the inequality $\|x\|\leq \frac{2\|x_0\|}{1-C_1}$ and this solution is unique among those such that $\|x\|<  \frac{1-C_1}{2C_2}.$
	\end{lemma}

\begin{lemma}[see \cite{Nati}]\label{lemacalor}
Assume that $a\geq0,\sigma\geq1,$ $T>0$, $s\in \mathbb{R}$, $\gamma \in \R$, $f\in L^1_T(\mathcal{X}_{a,\sigma}^{s}(\mathbb{R}^3))$ and $v_0\in \mathcal{X}_{a,\sigma}^{s}(\mathbb{R}^3)$. Consider that $v\in C_T(S'(\mathbb{R}^3))$ solves the system
\begin{align*}
\left\{
  \begin{array}{ll}
    v_t+(-\Delta)^{\gamma} v = f,\quad x\in \mathbb{R}^3,t\in(0,T];\\
    v(\cdot,0)=v_0,\quad x\in \mathbb{R}^3.
  \end{array}
\right.
\end{align*}
Then,  $$v\in C_T(\mathcal{X}_{a,\sigma}^{s}(\mathbb{R}^3))\cap L^p_T(\mathcal{X}_{a,\sigma}^{s+\frac{2\gamma}{p}}(\mathbb{R}^3)),\quad \forall p\geq1.$$ Furthermore, one can write the following inequalities:
\begin{enumerate}
  \item[\emph{i)}] $\|v\|_{L^\infty_T(\mathcal{X}_{a,\sigma}^{s})}\leq \|v_0\|_{\mathcal{X}_{a,\sigma}^{s}}+\|f\|_{L^1_T(\mathcal{X}_{a,\sigma}^{s})};$
  \item[\emph{ii)}] $\displaystyle \|v\|_{L^p_T(\mathcal{X}_{a,\sigma}^{s+\frac{2\gamma}{p}})}\leq       \|v_0\|_{\mathcal{X}_{a,\sigma}^{s}}+\|f\|_{L^1_T(\mathcal{X}_{a,\sigma}^{s})}.$
\end{enumerate}
\end{lemma}

\begin{lemma}[see \cite{Nata}]\label{lema2}
	Assume that $f$, $g\in \mathcal{X}_{a,\sigma}^{s+1}(\mathbb{R}^3)\cap\mathcal{X}_{\frac{a}{\sigma},\sigma}^{0}(\mathbb{R}^3)$, with $a\geq 0$, $\sigma\geq 1$ and $s\geq-1$. Then $fg\in \mathcal{X}_{a,\sigma}^{s+1}(\mathbb{R}^3)$. Moreover, the inequality below holds:
\begin{align*}
\|fg\|_{\mathcal{X}^{s+1}_{a,\sigma}}\leq C_s[\|f\|_{\mathcal{X}^{0}_{\frac{a}{\sigma},\sigma}}\|g\|_{\mathcal{X}^{s+1}_{a,\sigma}}+\|f\|_{\mathcal{X}^{s+1}_{a,\sigma}(\mathbb{R}^3)}\|g\|_{\mathcal{X}^{0}_{\frac{a}{\sigma},\sigma}}].
\end{align*}
\end{lemma}

\begin{lemma}[see \cite{Nata}]\label{lema1}
Let $a\geq 0$, $\gamma\geq\frac{1}{2}$ and $\sigma\geq 1$. The following inequalities hold:
\begin{enumerate}
	\item[\emph{i)}] $\|f\|_{\mathcal{X}^{s+1}_{a,\sigma}}\leq\|f\|_{\mathcal{X}^{s}_{a,\sigma}}^{1-\frac{1}{2\gamma}}\|f\|_{\mathcal{X}^{s+2\gamma}_{a,\sigma}}^{\frac{1}{2\gamma}},$ if  $s\in \mathbb{R}$;
	\item[\emph{ii)}] $\|f\|_{\mathcal{X}^{0}_{a,\sigma}}\leq\|f\|_{\mathcal{X}^{s}_{a,\sigma}}^{1+\frac{s}{2\gamma}}\|f\|_{\mathcal{X}^{s+2\gamma}_{a,\sigma}}^{-\frac{s}{2\gamma}}$, if $-2\gamma\leq s\leq 0$.
\end{enumerate}
\end{lemma}

\begin{lemma}[see \cite{artigowilberthyagomanasses}]\label{lemaprincipal}
	Let $a,\sigma$ and $s$ be real numbers such that $(a,s,\sigma)\in \{(0,\infty)$ $\times(-\infty,0)\times (1,\infty) \}\cup \{[0,\infty)\times\{0\}\times [1,\infty) \}$. Assume that $f\in \mathcal{X}_{a,\sigma}^{s}(\mathbb{R}^3)$. Then, $f\in \mathcal{X}_{\frac{a}{\sigma},\sigma}^{0}(\mathbb{R}^3)$. Moreover, there exists a positive constant $C_{a,s,\sigma}$ such that
	\begin{align*}
	\|f\|_{\mathcal{X}_{\frac{a}{\sigma},\sigma}^{0}}\leq C_{a,s,\sigma}\|f\|_{\mathcal{X}_{a,\sigma}^{s}}.
	\end{align*}
\end{lemma}

Now, allow us to present the lemmas which will be useful in the proofs of our blow-up criteria for the local solutions of Boussinesq equations (\ref{MHD-alpha}) (for more details, see \cite{Nata}).

\begin{lemma}[see \cite{Nata}]\label{lema0}
	Let $a>0$, $\sigma\geq 1$, $\mu>1$, $s\leq0$, $\delta\in\mathbb{R}$ and $f\in \mathcal{X}_{a,\sigma}^{s+\delta}(\mathbb{R}^3)$. Then, $f\in \mathcal{X}_{\frac{a}{\mu},\sigma}^{\delta}(\mathbb{R}^3)$. Moreover, there exists a positive constant $C_{a,s,\delta,\sigma,\mu}$ such that
	\begin{align*}
	\|f\|_{\mathcal{X}_{\frac{a}{\mu},\sigma}^{\delta}}\leq C_{a,s,\delta,\sigma,\mu}\|f\|_{\mathcal{X}_{a,\sigma}^{s+\delta}}.
	\end{align*}
\end{lemma}

\begin{lemma}[see \cite{Nata}]\label{lema}
	Let $\delta>0$ such that $f\in L^2(\mathbb{R}^3)\cap \mathcal{X}^\delta(\mathbb{R}^3)$. Then, $f\in \mathcal{X}^0(\mathbb{R}^3).$ Moreover, the following inequality holds:
	\begin{align*}
	\|f\|_{\mathcal{X}^0}\leq C_0 \|f\|_{L^2}^{\frac{2\delta}{2\delta+3}}\|f\|_{\mathcal{X}^\delta}^{\frac{3}{2\delta+3}}.
	\end{align*}
\end{lemma}

\begin{lemma}[see \cite{Nata}]\label{lemanovo2}
	Let  $a>0$, $\sigma\geq1$,  $s,\delta\in \mathbb{R}$ such that $s\leq\delta$ and $f\in \mathcal{X}_{a,\sigma}^{s}(\mathbb{R}^3)$. Then, $f\in\mathcal{X}^{\delta}(\mathbb{R}^3)$. Moreover, there is a positive constant $C_{a,s,\delta,\sigma}$ such that
	\begin{align*}
	\|f\|_{\mathcal{X}^{\delta}}\leq C_{a,s,\delta,\sigma}\|f\|_{\mathcal{X}_{a,\sigma}^{s}}.
	\end{align*}
\end{lemma}

\section{Local theory}\label{secaolocalB}



This section is devoted to establishing the proofs of Theorems \ref{teoremaexistenciaB2} and \ref{teoremaexistenciaB}. These results are related to the local existence and uniqueness of mild solutions to the Boussinesq equations (\ref{MHD-alpha}).

\bigskip
\noindent\textbf{Proof of Theorem \ref{teoremaexistenciaB2}:}
By using the mild solution representation given by \eqref{Mild SOlution Form}, we obtain
\begin{align}\label{estimativaprojecao3}
(u,\theta)(t)&= (e^{- t(-\Delta)^{\alpha}}u_0,e^{- t (-\Delta)^{\beta}}\theta_0) + B((u,\theta),(u,\theta))(t)+L(u,\theta)(t),\quad\forall t>0,
\end{align}
where
\begin{align*}
B((w,v),(\gamma,\phi))(t)=(B_1((w,v),(\gamma,\phi)),B_2((w,v),(\gamma,\phi)))(t)
\end{align*}
and also
\begin{align}\label{linearL}
L(w,v)(t)=(L_1(w,v)(t),L_2(w,v)(t)),
\end{align}
with
\begin{align}\label{bilinearB1}
B_1((w,v),(\gamma,\phi))(t)= - \int_{0}^te^{- (t-\tau)(-\Delta)^{\alpha}} \mathbb{P}(w\cdot\nabla \gamma)\,d\tau,
\end{align}
\begin{align}\label{bilinearB2}
B_2((w,v),(\gamma,\phi))(t)= - \int_{0}^te^{- (t-\tau)(-\Delta)^{\beta}} [w\cdot \nabla \phi]\,d\tau,
\end{align}
\begin{align}\label{linearL1eL2}
L_1(w,v)(t)= \int_{0}^te^{- (t-\tau)(-\Delta)^{\alpha}} \mathbb{P}[ve_3]\,d\tau \quad \hbox{    and    }\quad  L_2(w,v)(t)= 0,
\end{align}
for all $(w,v),(\gamma,\phi)\in X\times Y:= [C_T(\mathcal{X}_{a,\sigma}^s)\cap L_T(\mathcal{X}_{a,\sigma}^{s+2\alpha})]\times [C_T(\mathcal{X}_{a,\sigma}^s)\cap L_T(\mathcal{X}_{a,\sigma}^{s+2\beta})]\footnote{Here, $\|(f,g)\|_{X\times Y}:=\|f\|_X+\|g\|_Y,$ where $\|f\|_X=\|f\|_{L^\infty_T(\mathcal{X}_{a,\sigma}^s)}+\|f\|_{L^1_T(\mathcal{X}_{a,\sigma}^{s+2\alpha})}$, and $\|g\|_Y=\|g\|_{L^\infty_T(\mathcal{X}_{a,\sigma}^s)}+\|g\|_{L^1_T(\mathcal{X}_{a,\sigma}^{s+2\beta})}$}$ ($T>0$ will be chosen as follows), where $\hbox{div}\,w=\hbox{div}\,\gamma=0$ (for compatibility purposes with $u$ in (\ref{MHD-alpha})). It is easy to check  that $B$ and $L$ are  bilinear and linear operators on $[X\times Y]^2$ and $X\times Y$, respectively. 
Thus, in order to apply Lemma \ref{lemaB}, we need to prove that $L$ and $B$ are continuous operators.   

We are going to begin  studying the operator $L$. At first, notice that (\ref{linearL1eL2}) infers the system
\begin{align}\label{sistemaL}
\left\{
  \begin{array}{ll}
    \partial_t L_1(w,v)(t)+(-\Delta)^{\alpha} L_1(w,v)(t) = \mathbb{P}(ve_3)(t);\\
    L_1(w,v)(0)=0.
  \end{array}
\right.
\end{align}
and, moreover, one has
 \begin{align}\label{p3}
\nonumber\displaystyle\|\mathbb{P}(ve_3)\|_{L^1_T(\mathcal{X}_{a,\sigma}^{s})}&\leq  \int_0^T\int_{\mathbb{R}^3} |\xi|^{s}e^{a|\xi|^{\frac{1}{\sigma}}} |\mathcal{F}[\mathbb{P}( ve_3)](\xi)|\,d\xi\,dt\leq\int_0^T\int_{\mathbb{R}^3} |\xi|^{s} e^{a|\xi|^{\frac{1}{\sigma}}} |\widehat{v}(\xi)|\,d\xi\,dt\\
&\leq T \|v\|_{L^{\infty}_T(\mathcal{X}_{a,\sigma}^s)}\leq T \|v\|_{Y}.
\end{align}
By Lemma \ref{lemacalor} (with $p=1$) and (\ref{sistemaL}), it follows that
\begin{align}\label{wilber18}
\displaystyle\|L_1(w,v)\|_{X}& \leq 2T \|v\|_{Y},\quad\forall (w,v)\in X\times Y.
\end{align}
As a result,  (\ref{linearL}), (\ref{linearL1eL2}) and (\ref{wilber18}) imply that
\begin{align}\label{lwfinal}
\displaystyle\|L(w,v)\|_{X\times Y}\leq  2T \|(w,v)\|_{X\times Y},\quad\forall (w,v)\in X\times Y.
\end{align}
This leads us to observe that $L$ is a continuous operator. 

Now let us estimate the operator $B$. Notice that (\ref{bilinearB2}) allows us to write the next system:
\begin{align}\label{sistemaB2}
\left\{
  \begin{array}{ll}
    \partial_t B_2((w,v),(\gamma,\phi))(t) + (-\Delta)^{\beta} B_2((w,v),(\gamma,\phi))(t) = -w\cdot\nabla \phi(t); \\
    B_2((w,v),(\gamma,\phi))(0)=0.
  \end{array}
\right.
\end{align}
On the other hand, one concludes
\begin{align*}
\nonumber\|w\cdot\nabla \phi\|_{L^1_T(\mathcal{X}_{a,\sigma}^{s})}
&=\int_0^T \int_{\mathbb{R}^3}|\xi|^{s}e^{a|\xi|^{\frac{1}{\sigma}}}|\mathcal{F}[w\cdot\nabla \phi](t)| \,d\xi dt
\leq\int_0^T \int_{\mathbb{R}^3}|\xi|^{s+1}e^{a|\xi|^{\frac{1}{\sigma}}}|\mathcal{F}[\phi w](t)| \,d\xi dt\\
&=\int_0^T \|(\phi w) (t) \|_{\mathcal{X}_{a,\sigma}^{s+1}}\,dt.
\end{align*}
Consequently,  Lemma \ref{lema2} and Lemma \ref{lema1}  can be applied and we obtain
\begin{align}\label{wp1}
\nonumber\|w\cdot\nabla \phi\|_{L^1_T(\mathcal{X}_{a,\sigma}^{s})}&\leq C_{s}\int_0^T[ \|\phi\|_{\mathcal{X}_{\frac{a}{\sigma},\sigma}^{0}}\|w \|_{\mathcal{X}_{a,\sigma}^{s+1}}+ \|\phi\|_{\mathcal{X}_{a,\sigma}^{s+1}}\|w \|_{\mathcal{X}_{\frac{a}{\sigma},\sigma}^{0}}]\,dt\\
\nonumber	&\leq C_{s}\|\phi \|_{L^\infty_T(\mathcal{X}_{a,\sigma}^{s})}^{1+\frac{s}{2\beta}}\|w \|_{L^\infty_T(\mathcal{X}_{a,\sigma}^{s})}^{1-\frac{1}{2\alpha}}\int_0^T \|\phi\|_{\mathcal{X}^{s+2\beta}_{a,\sigma}}^{\frac{-s}{2\beta}}\|w\|_{\mathcal{X}^{s+2\alpha}_{a,\sigma}}^{\frac{1}{2\alpha}}\,dt\\
&\quad+C_{s}\|w \|_{L^\infty_T(\mathcal{X}_{a,\sigma}^{s})}^{1+\frac{s}{2\alpha}}\|\phi \|_{L^\infty_T(\mathcal{X}_{a,\sigma}^{s})}^{1-\frac{1}{2\beta}}\int_0^T \|w\|_{\mathcal{X}^{s+2\alpha}_{a,\sigma}}^{\frac{-s}{2\alpha}} \|\phi\|_{\mathcal{X}^{s+2\beta}_{a,\sigma}}^{\frac{1}{2\beta}}\,dt,
\end{align}
provided that $a\geq0$, $\sigma\geq1$, $\max\{-1,-2\alpha, -2\beta\}\leq s \leq0$, $\alpha\geq\frac{1}{2}$ and $\beta\geq\frac{1}{2}$. By using  H\"older's inequality twice, we can write
\begin{align}\label{wp4}
\nonumber\|w\cdot\nabla \phi\|_{L^1_T(\mathcal{X}_{a,\sigma}^{s})} \nonumber&\leq C_{s}\|\phi \|_{L^\infty_T(\mathcal{X}_{a,\sigma}^{s})}^{1+\frac{s}{2\beta}}\|w \|_{L^\infty_T(\mathcal{X}_{a,\sigma}^{s})}^{1-\frac{1}{2\alpha}}\|\phi\|_{L^1_T(\mathcal{X}^{s+2\beta}_{a,\sigma})}^{\frac{-s}{2\beta}}\Big(\int_0^T \|w\|_{\mathcal{X}^{s+2\alpha}_{a,\sigma}}^{\frac{\beta}{\alpha(s+2\beta)}}\,dt\Big)^{1+\frac{s}{2\beta}}\\
\nonumber&\quad+C_{s}\|w \|_{L^\infty_T(\mathcal{X}_{a,\sigma}^{s})}^{1+\frac{s}{2\alpha}}\|\phi \|_{L^\infty_T(\mathcal{X}_{a,\sigma}^{s})}^{1-\frac{1}{2\beta}} \|w\|_{L^1_T(\mathcal{X}^{s+2\alpha}_{a,\sigma})}^{\frac{-s}{2\alpha}} \Big(\int_0^T \|\phi\|_{\mathcal{X}^{s+2\beta}_{a,\sigma}}^{\frac{\alpha}{\beta(s+2\alpha)}}\,dt\Big)^{1+\frac{s}{2\alpha}}\\
\nonumber&\leq C_{s}T^{1+\frac{s}{2\beta}-\frac{1}{2\alpha}}\|\phi \|_{L^\infty_T(\mathcal{X}_{a,\sigma}^{s})}^{1+\frac{s}{2\beta}}\|w \|_{L^\infty_T(\mathcal{X}_{a,\sigma}^{s})}^{1-\frac{1}{2\alpha}}\|\phi\|_{L^1_T(\mathcal{X}^{s+2\beta}_{a,\sigma})}^{\frac{-s}{2\beta}} \|w\|_{L^1_T(\mathcal{X}^{s+2\alpha}_{a,\sigma})}^{\frac{1}{2\alpha}}\\
&\quad+C_{s}T^{1+\frac{s}{2\alpha}-\frac{1}{2\beta}}\|w \|_{L^\infty_T(\mathcal{X}_{a,\sigma}^{s})}^{1+\frac{s}{2\alpha}}\|\phi \|_{L^\infty_T(\mathcal{X}_{a,\sigma}^{s})}^{1-\frac{1}{2\beta}} \|w\|_{L^1_T(\mathcal{X}^{s+2\alpha}_{a,\sigma})}^{\frac{-s}{2\alpha}}  \|\phi\|_{L^1_T(\mathcal{X}^{s+2\beta}_{a,\sigma})}^{\frac{1}{2\beta}},
	\end{align}
provided that $s\geq \max\{\frac{\beta(1-2\alpha)}{\alpha},\frac{\alpha(1-2\beta)}{\beta}\}$. Hence, we establish the following inequality:
\begin{align}\label{existencianova2n}
	\|w\cdot\nabla \phi\|_{L^1_T(\mathcal{X}_{a,\sigma}^{s})}&\leq C_{s}[T^{1+\frac{s}{2\beta}-\frac{1}{2\alpha}}+T^{1+\frac{s}{2\alpha}-\frac{1}{2\beta}}]\|w \|_{X}\|\phi\|_{Y},\quad\forall  (w,\phi) \in X\times Y.
\end{align}
Therefore, by Lemma \ref{lemacalor}  (with $p=1$), (\ref{sistemaB2}) and (\ref{existencianova2n}), one concludes
\begin{align}\label{desigualdade4n}
	\|B_2((w,v),(\gamma,\phi))\|_{Y}
&\leq C_{s}[T^{1+\frac{s}{2\beta}-\frac{1}{2\alpha}}+T^{1+\frac{s}{2\alpha}-\frac{1}{2\beta}}]\|(w,v) \|_{X\times Y}\|(\gamma,\phi)\|_{X\times Y},\quad\forall  (w,v),(\gamma,\phi) \in X\times Y.
\end{align}

By following similar arguments, we are able to estimate $B_1$. In fact,
by using the definition (\ref{bilinearB1}), one can write
\begin{align}\label{sistemaB1}
\left\{
  \begin{array}{ll}
    \partial_t B_1((w,v),(\gamma,\phi))(t) + (-\Delta)^{\alpha} B_1((w,v),(\gamma,\phi))(t) = -\mathbb{P}[w\cdot\nabla \gamma](t); \\
    B_1((w,v),(\gamma,\phi))(0)=0.
  \end{array}
\right.
\end{align}
On the other hand, in order to apply Lemma \ref{lemacalor}, notice that
\begin{align}\label{p8}
\|\mathbb{P}[w\cdot\nabla \gamma]\|_{L^1_T(\mathcal{X}_{a,\sigma}^{s})}&
\leq\int_0^T \|(\gamma\otimes w) (t) \|_{\mathcal{X}_{a,\sigma}^{s+1}}\,dt.
\end{align}
As a result, we apply  Lemmas \ref{lema2} and  \ref{lema1} to reach the inequalities below:
\begin{align}\label{wp3}
\nonumber\|\mathbb{P}[w\cdot\nabla \gamma]\|_{L^1_T(\mathcal{X}_{a,\sigma}^{s})} \nonumber	&\leq C_{s}\|\gamma\|_{L^\infty_T(\mathcal{X}_{a,\sigma}^{s})}^{1+\frac{s}{2\alpha}}\|w \|_{L^\infty_T(\mathcal{X}_{a,\sigma}^{s})}^{1-\frac{1}{2\alpha}}\int_0^T \|\gamma\|_{\mathcal{X}^{s+2\alpha}_{a,\sigma}}^{\frac{-s}{2\alpha}}\|w\|_{\mathcal{X}^{s+2\alpha}_{a,\sigma}}^{\frac{1}{2\alpha}}\,dt\\
&\quad+C_{s}\|w \|_{L^\infty_T(\mathcal{X}_{a,\sigma}^{s})}^{1+\frac{s}{2\alpha}}\|\gamma \|_{L^\infty_T(\mathcal{X}_{a,\sigma}^{s})}^{1-\frac{1}{2\alpha}}\int_0^T \|w\|_{\mathcal{X}^{s+2\alpha}_{a,\sigma}}^{\frac{-s}{2\alpha}} \|\gamma\|_{\mathcal{X}^{s+2\alpha}_{a,\sigma}}^{\frac{1}{2\alpha}}\,dt,
\end{align}
provided that $a\geq0$, $\sigma\geq1$, $\max\{-1,-2\alpha\}\leq s \leq0$ and $\alpha\geq\frac{1}{2}$. As a consequence,  by   H\"older's inequality, it is true that
\begin{align}\label{wp5}
\nonumber\|\mathbb{P}[w\cdot\nabla \gamma]\|_{L^1_T(\mathcal{X}_{a,\sigma}^{s})} 
&\leq C_{s}T^{1+\frac{s-1}{2\alpha}}\|\gamma \|_{L^\infty_T(\mathcal{X}_{a,\sigma}^{s})}^{1+\frac{s}{2\alpha}}\|w \|_{L^\infty_T(\mathcal{X}_{a,\sigma}^{s})}^{1-\frac{1}{2\alpha}}\|\gamma\|_{L^1_T(\mathcal{X}^{s+2\alpha}_{a,\sigma})}^{\frac{-s}{2\alpha}} \|w\|_{L^1_T(\mathcal{X}^{s+2\alpha}_{a,\sigma})}^{\frac{1}{2\alpha}}\\
&\quad+C_{s}T^{1+\frac{s-1}{2\alpha}}\|w \|_{L^\infty_T(\mathcal{X}_{a,\sigma}^{s})}^{1+\frac{s}{2\alpha}}\|\gamma \|_{L^\infty_T(\mathcal{X}_{a,\sigma}^{s})}^{1-\frac{1}{2\alpha}} \|w\|_{L^1_T(\mathcal{X}^{s+2\alpha}_{a,\sigma})}^{\frac{-s}{2\alpha}}  \|\gamma\|_{L^1_T(\mathcal{X}^{s+2\alpha}_{a,\sigma})}^{\frac{1}{2\alpha}},
	\end{align}
provided that $s\geq 1-2\alpha$. Then, the next inequality holds:
\begin{align}\label{existencianova2n1}
	\|\mathbb{P}[w\cdot\nabla \gamma]\|_{L^1_T(\mathcal{X}_{a,\sigma}^{s})}&\leq C_{s}T^{1+\frac{s-1}{2\alpha}}\|w \|_{X}\|\gamma\|_{X},\quad\forall  w,\gamma \in X.
\end{align}
Thereby, Lemma \ref{lemacalor}  (with $p=1$), (\ref{sistemaB1}) and (\ref{existencianova2n1}) allow us to conclude that
\begin{align}\label{desigualdade4n1}
	\|B_1((w,v),(\gamma,\phi))\|_{X}
&\leq C_{s}T^{1+\frac{s-1}{2\alpha}}\|(w,v) \|_{X\times Y}\|(\gamma,\phi)\|_{X\times Y},\quad\forall  (w,v),(\gamma,\phi) \in X\times Y.
\end{align}
The estimates (\ref{desigualdade4n}) and (\ref{desigualdade4n1}) prove that the operator $B$ is continuous and, furthermore, it holds 
\begin{align}\label{estimativaB}
	\|B((w,v),(\gamma,\phi))\|_{X\times Y}
&\leq C_{s}[T^{1+\frac{s}{2\beta}-\frac{1}{2\alpha}}+T^{1+\frac{s}{2\alpha}-\frac{1}{2\beta}}+T^{1+\frac{s-1}{2\alpha}}]\|(w,v) \|_{X\times Y}\|(\gamma,\phi)\|_{X\times Y},
\end{align}
for all $(w,v),(\gamma,\phi) \in X\times Y.$

It remains to estimate the first term on the right-hand side of (\ref{estimativaprojecao3}). Thus, note that
\begin{align*}
	\|e^{- t(-\Delta)^{\beta}}\theta_0\|_{\mathcal{X}_{a,\sigma}^{s}}
&= \int_{\mathbb{R}^3} |\xi|^{s} e^{a|\xi|^{\frac{1}{\sigma}}}e^{- t|\xi|^{2\beta}} |\widehat{\theta}_0(\xi)|\,d\xi\leq \int_{\mathbb{R}^3} |\xi|^{s} e^{a|\xi|^{\frac{1}{\sigma}}} |\widehat{\theta}_0(\xi)|\,d\xi,
\end{align*}
for all $t\in[0,T]$. This means that
\begin{align}\label{novo1}
	\|e^{- t(-\Delta)^{\beta}}\theta_0\|_{L^\infty_T(\mathcal{X}_{a,\sigma}^{s})}&\leq \|\theta_0\|_{\mathcal{X}_{a,\sigma}^{s}}.
\end{align}
On the other hand, it also holds
\begin{align}\label{novo2}
	\nonumber \|e^{- t(-\Delta)^{\beta}}\theta_0\|_{L^{1}_T(\mathcal{X}_{a,\sigma}^{s+2\beta})}
&= \int_{\mathbb{R}^3} |\xi|^{s+2\beta} e^{a|\xi|^{\frac{1}{\sigma}}} |\widehat{\theta}_0(\xi)|\Big(\int_0^T e^{- t|\xi|^{2\beta}}\,dt\Big)\,d\xi\\
&\leq \int_{\mathbb{R}^3} |\xi|^{s} e^{a|\xi|^{\frac{1}{\sigma}}} |\widehat{\theta}_0(\xi)|\,d\xi= \|\theta_0\|_{\mathcal{X}_{a,\sigma}^{s}}.
\end{align}
By putting (\ref{novo1}) and (\ref{novo2}) together, we obtain
\begin{align}\label{desigualdade1np}
	\|e^{- t(-\Delta)^{\beta}}\theta_0\|_{Y}\leq 2\|\theta_0\|_{\mathcal{X}_{a,\sigma}^{s}}.
\end{align}
Similarly, we obtain
\begin{align}\label{desigualdade1n1}
	\|e^{- t(-\Delta)^{\alpha}}u_0\|_{X}\leq 2\|u_0\|_{\mathcal{X}_{a,\sigma}^{s}}.
\end{align}
Hence, from (\ref{desigualdade1np}) and (\ref{desigualdade1n1}), one infers
\begin{align}\label{desigualdade1n}
	\|(e^{- t(-\Delta)^{\alpha}}u_0,e^{- t(-\Delta)^{\beta}}\theta_0)\|_{X\times Y}\leq 2\|(u_0,\theta_0)\|_{\mathcal{X}_{a,\sigma}^{s}}.
\end{align}

Now, assume that $\rho\in[0,1)$ is such that
\begin{align}\label{rho}
\rho=\min\Big\{1+\frac{s}{2\beta}-\frac{1}{2\alpha},1+\frac{s}{2\alpha}-\frac{1}{2\beta},1+\frac{s-1}{2\alpha}\Big\},
\end{align}
provided $\max\{\frac{\beta(1-2\alpha)}{\alpha},\frac{\alpha(1-2\beta)}{\beta},1-2\alpha\}\leq s\leq0$, in order to split this proof into two cases.\\\\
\bigskip
\noindent \underline{Case 1}: Consider that 
\begin{align}\label{p10}
s> \max\Big\{1-2\alpha,\frac{\beta(1-2\alpha)}{\alpha},\frac{\alpha(1-2\beta)}{\beta}\Big\}\hbox{  and  } 0<T< [2+(24C_s\|(u_0,\theta_0)\|_{\mathcal{X}_{a,\sigma}^{s}})^{\frac{1}{2}}]^{-\frac{2}{\rho}}, 
\end{align}
where $C_s$ is given in (\ref{estimativaB}).
In this case, it is worth to notice that $0<T<\frac{1}{2}$, $\alpha>\frac{1}{2}$, $\beta>\frac{1}{2}$ and $0<\rho<1$ (see (\ref{rho})). Therefore, by (\ref{desigualdade1n}) and (\ref{p10}), the following inequalities are true:
\begin{align}\label{p1}
	\nonumber &4C_s[T^{1+\frac{s}{2\beta}-\frac{1}{2\alpha}}+T^{1+\frac{s}{2\alpha}-\frac{1}{2\beta}}+T^{1+\frac{s-1}{2\alpha}}]\|(e^{- t(-\Delta)^{\alpha}}u_0,e^{- t(-\Delta)^{\beta}}\theta_0)\|_{X\times Y}\\
&\leq 24C_sT^\rho\|(u_0,\theta_0)\|_{\mathcal{X}_{a,\sigma}^{s}}<(1-2T)^2.
\end{align}

\bigskip
\noindent \underline{Case 2}: Assume that 
\begin{align}\label{p11}
s= \max\Big\{\frac{\beta(1-2\alpha)}{\alpha},\frac{\alpha(1-2\beta)}{\beta},1-2\alpha\Big\},\quad \|(u_0,\theta_0)\|_{\mathcal{X}_{a,\sigma}^{s}}<[24C_s]^{-1} 
\end{align}
and also  
\begin{align}\label{p12}
0<T< \frac{1}{2}-[6C_s\|(u_0,\theta_0)\|_{\mathcal{X}_{a,\sigma}^{s}}]^{\frac{1}{2}},
\end{align}
 where $C_s$ is given in (\ref{estimativaB}).
It is important to point out that  $0<T<\frac{1}{2}$, $\alpha\geq\frac{1}{2}$, $\beta\geq \frac{1}{2}$ and $\rho=0$ (see (\ref{rho})). Consequently, by (\ref{desigualdade1n}), (\ref{rho}), (\ref{p11}) and (\ref{p12}), one can write
\begin{align}\label{p2}
	\nonumber& 4C_s[T^{1+\frac{s}{2\beta}-\frac{1}{2\alpha}}+T^{1+\frac{s}{2\alpha}-\frac{1}{2\beta}}+T^{1+\frac{s-1}{2\alpha}}]\|(e^{- t(-\Delta)^{\alpha}}u_0,e^{- t(-\Delta)^{\beta}}\theta_0)\|_{X\times Y}\\
&\leq 24C_s\|(u_0,\theta_0)\|_{\mathcal{X}_{a,\sigma}^{s}}<(1-2T)^2.
\end{align}

In these two cases above, by (\ref{estimativaprojecao3}), (\ref{lwfinal}), (\ref{estimativaB}), (\ref{p1}), (\ref{p2}) and Lemma \ref{lemaB}, there exists a unique $(u,\theta)\in X\times Y$ such that
$$\|(u,\theta)\|_{X\times Y}< \frac{1-2T}{2C_s[T^{1+\frac{s}{2\beta}-\frac{1}{2\alpha}}+T^{1+\frac{s}{2\alpha}-\frac{1}{2\beta}}+T^{1+\frac{s-1}{2\alpha}}]},$$
and, furthermore, 
\begin{align*}
\|(u,\theta)\|_{X\times Y}\leq \frac{2\|(e^{- t(-\Delta)^{\alpha}}u_0,e^{- t(-\Delta)^{\beta}}\theta_0)\|_{X\times Y}}{1-2T}\leq \frac{4\|(u_0,\theta_0)\|_{\mathcal{X}_{a,\sigma}^{s}}}{1-2T}.
\end{align*}

It is relevant to emphasize that the two first equations in the Boussinesq system (\ref{MHD-alpha}) can be rewritten  as follows:
\begin{equation}\label{sistemauteta}
\left\{
\begin{array}{l}
u_t
+
(-\Delta)^{\alpha}u
=
\mathbb{P}(\theta e_3)-\mathbb{P}(u\cdot\nabla u);\\
\theta_t
+
(-\Delta)^{\beta}\theta
= - u\cdot\nabla\theta.
\end{array}
\right.
\end{equation}
We shall apply Lemma \ref{lemacalor} to this system (\ref{sistemauteta}). To this end, similarly to (\ref{p3}), (\ref{existencianova2n1}) and (\ref{existencianova2n}), notice that
 \begin{align}\label{p4}
\displaystyle\|\mathbb{P}(\theta e_3)\|_{L^1_T(\mathcal{X}_{a,\sigma}^{s})}&\leq  \int_0^T\int_{\mathbb{R}^3} |\xi|^{s}e^{a|\xi|^{\frac{1}{\sigma}}} |\widehat{\theta}(\xi)|\,d\xi\,dt
\leq T \|\theta\|_{Y}<\infty,
\end{align}
 \begin{align}\label{p5}
	\|\mathbb{P}[u\cdot\nabla u]\|_{L^1_T(\mathcal{X}_{a,\sigma}^{s})}&\leq C_{s}T^{1+\frac{s-1}{2\alpha}}\|u \|_{X}^2<\infty
\end{align}
and also
\begin{align}\label{p6}
	\|u\cdot\nabla \theta\|_{L^1_T(\mathcal{X}_{a,\sigma}^{s})}&\leq C_{s}[T^{1+\frac{s}{2\beta}-\frac{1}{2\alpha}}+T^{1+\frac{s}{2\alpha}-\frac{1}{2\beta}}]\|u \|_{X}\|\theta\|_{Y}<\infty,
\end{align}
since $(u,\theta)\in X\times Y.$ Therefore, (\ref{sistemauteta}), (\ref{p4}), (\ref{p5}), (\ref{p6}) and Lemma \ref{lemacalor} imply that
$$(u,\theta)\in L^p_{T}(\mathcal{X}_{a,\sigma}^{s+\frac{2\alpha}{p}}(\mathbb{R}^3))\times L^p_{T}(\mathcal{X}_{a,\sigma}^{s+\frac{2\beta}{p}}(\mathbb{R}^3)),\quad \forall p\geq1.$$
\caixa

\noindent\textbf{Proof of Theorem \ref{teoremaexistenciaB}:}
We shall consider, as in the proof of Theorem \ref{teoremaexistenciaB2}, that (\ref{estimativaprojecao3})-(\ref{linearL1eL2}) hold  in the same function space. Thus, let us show another way to establish that  $B$ is a continuous operator, since the estimate (\ref{lwfinal}) will be reapplied as it is. Thus, by Lemmas \ref{lema2}, \ref{lemaprincipal} and \ref{lema1}  i), it is true that
\begin{align}\label{wt2}
\nonumber\|w\cdot\nabla \phi\|_{L^1_T(\mathcal{X}_{a,\sigma}^{s})}&\leq C_{s}\int_0^T[ \|\phi\|_{\mathcal{X}_{\frac{a}{\sigma},\sigma}^{0}}\|w \|_{\mathcal{X}_{a,\sigma}^{s+1}}+ \|\phi\|_{\mathcal{X}_{a,\sigma}^{s+1}}\|w \|_{\mathcal{X}_{\frac{a}{\sigma},\sigma}^{0}}]\,dt\\
	\nonumber &\leq C_{a,\sigma,s}\|\phi \|_{L^\infty_T(\mathcal{X}_{a,\sigma}^{s})}\|w \|_{L^\infty_T(\mathcal{X}_{a,\sigma}^{s})}^{1-\frac{1}{2\alpha}}\int_0^T \|w\|_{\mathcal{X}^{s+2\alpha}_{a,\sigma}}^{\frac{1}{2\alpha}}\,dt\\
&\quad+C_{a,\sigma,s}\|w \|_{L^\infty_T(\mathcal{X}_{a,\sigma}^{s})}\|\phi \|_{L^\infty_T(\mathcal{X}_{a,\sigma}^{s})}^{1-\frac{1}{2\beta}}\int_0^T  \|\phi\|_{\mathcal{X}^{s+2\beta}_{a,\sigma}}^{\frac{1}{2\beta}}\,dt,
\end{align}
provided that $(a,s,\sigma)\in \{(0,\infty)$ $\times[-1,0)\times (1,\infty) \}\cup \{[0,\infty)\times\{0\}\times [1,\infty) \}$, $\alpha\geq\frac{1}{2}$ and $\beta\geq\frac{1}{2}$. By applying  H\"older's inequality, one reaches
\begin{align}\label{wt3}
\nonumber\|w\cdot\nabla \phi\|_{L^1_T(\mathcal{X}_{a,\sigma}^{s})} \nonumber
&\leq C_{a,\sigma,s}T^{1-\frac{1}{2\alpha}}\|\phi \|_{L^\infty_T(\mathcal{X}_{a,\sigma}^{s})}\|w \|_{L^\infty_T(\mathcal{X}_{a,\sigma}^{s})}^{1-\frac{1}{2\alpha}} \|w\|_{L^1_T(\mathcal{X}^{s+2\alpha}_{a,\sigma})}^{\frac{1}{2\alpha}}\\
&\quad+C_{a,\sigma,s}T^{1-\frac{1}{2\beta}}\|w \|_{L^\infty_T(\mathcal{X}_{a,\sigma}^{s})}\|\phi \|_{L^\infty_T(\mathcal{X}_{a,\sigma}^{s})}^{1-\frac{1}{2\beta}}  \|\phi\|_{L^1_T(\mathcal{X}^{s+2\beta}_{a,\sigma})}^{\frac{1}{2\beta}}.
	\end{align}
Thereby, we can conclude
\begin{align}\label{existencianova2np}
	\|w\cdot\nabla \phi\|_{L^1_T(\mathcal{X}_{a,\sigma}^{s})}&\leq C_{a,\sigma,s}[T^{1-\frac{1}{2\alpha}}+T^{1-\frac{1}{2\beta}}]\|w \|_{X}\|\phi\|_{Y},\quad\forall  (w,\phi) \in X\times Y.
\end{align}
Apply Lemma \ref{lemacalor}  (with $p=1$), (\ref{sistemaB2}) and (\ref{existencianova2np}) in order to obtain
\begin{align}\label{desigualdade4np}
	\|B_2((w,v),(\gamma,\phi))\|_{Y}
&\leq C_{a,\sigma,s}[T^{1-\frac{1}{2\alpha}}+T^{1-\frac{1}{2\beta}}]\|(w,v) \|_{X\times Y}\|(\gamma,\phi)\|_{X\times Y},\quad\forall  (w,v),(\gamma,\phi) \in X\times Y.
\end{align}
On the other hand, from (\ref{p8}) and Lemma \ref{lema2}, one obtains 
\begin{align*}
\nonumber\|\mathbb{P}[w\cdot\nabla \gamma]\|_{L^1_T(\mathcal{X}_{a,\sigma}^{s})}&
\leq C_{s}\int_0^T[ \|\gamma\|_{\mathcal{X}_{\frac{a}{\sigma},\sigma}^{0}}\|w \|_{\mathcal{X}_{a,\sigma}^{s+1}}+ \|\gamma\|_{\mathcal{X}_{a,\sigma}^{s+1}}\|w \|_{\mathcal{X}_{\frac{a}{\sigma},\sigma}^{0}}]\,dt,
\end{align*}
provided that $a\geq0$, $\sigma\geq1$ and $s\geq -1$. Consequently,  Lemmas \ref{lemaprincipal} and  \ref{lema1} i) imply
\begin{align}\label{wp2}
\nonumber\|\mathbb{P}[w\cdot\nabla \gamma]\|_{L^1_T(\mathcal{X}_{a,\sigma}^{s})} \nonumber	&\leq C_{a,\sigma,s}\|\gamma\|_{L^\infty_T(\mathcal{X}_{a,\sigma}^{s})}\|w \|_{L^\infty_T(\mathcal{X}_{a,\sigma}^{s})}^{1-\frac{1}{2\alpha}}\int_0^T\|w\|_{\mathcal{X}^{s+2\alpha}_{a,\sigma}}^{\frac{1}{2\alpha}}\,dt\\
&\quad+C_{a,\sigma,s}\|w \|_{L^\infty_T(\mathcal{X}_{a,\sigma}^{s})}\|\gamma \|_{L^\infty_T(\mathcal{X}_{a,\sigma}^{s})}^{1-\frac{1}{2\alpha}}\int_0^T  \|\gamma\|_{\mathcal{X}^{s+2\alpha}_{a,\sigma}}^{\frac{1}{2\alpha}}\,dt,
\end{align}
provided that $(a,s,\sigma)\in \{(0,\infty)$ $\times(-\infty,0)\times (1,\infty) \}\cup \{[0,\infty)\times\{0\}\times [1,\infty) \}$ and $\alpha\geq\frac{1}{2}$. Then,  by   H\"older's inequality, one can write 
\begin{align}\label{wt1}
\nonumber\|\mathbb{P}[w\cdot\nabla \gamma]\|_{L^1_T(\mathcal{X}_{a,\sigma}^{s})} 
&\leq C_{a,\sigma,s}T^{1-\frac{1}{2\alpha}}\|\gamma \|_{L^\infty_T(\mathcal{X}_{a,\sigma}^{s})}\|w \|_{L^\infty_T(\mathcal{X}_{a,\sigma}^{s})}^{1-\frac{1}{2\alpha}} \|w\|_{L^1_T(\mathcal{X}^{s+2\alpha}_{a,\sigma})}^{\frac{1}{2\alpha}}\\
&\quad+C_{a,\sigma,s}T^{1-\frac{1}{2\alpha}}\|w \|_{L^\infty_T(\mathcal{X}_{a,\sigma}^{s})}\|\gamma \|_{L^\infty_T(\mathcal{X}_{a,\sigma}^{s})}^{1-\frac{1}{2\alpha}}   \|\gamma\|_{L^1_T(\mathcal{X}^{s+2\alpha}_{a,\sigma})}^{\frac{1}{2\alpha}}.
	\end{align}
As a consequence, it follows that
\begin{align}\label{existencianova2n1p}
	\|\mathbb{P}[w\cdot\nabla \gamma]\|_{L^1_T(\mathcal{X}_{a,\sigma}^{s})}&\leq C_{a,\sigma,s}T^{1-\frac{1}{2\alpha}}\|w \|_{X}\|\gamma\|_{X},\quad\forall  w,\gamma \in X.
\end{align}
Hence, by Lemma \ref{lemacalor}  (with $p=1$), (\ref{sistemaB1}) and (\ref{existencianova2n1p}), we have
\begin{align}\label{desigualdade4n1p}
	\|B_1((w,v),(\gamma,\phi))\|_{X}
&\leq C_{a,\sigma,s}T^{1-\frac{1}{2\alpha}}\|(w,v) \|_{X\times Y}\|(\gamma,\phi)\|_{X\times Y},\quad\forall  (w,v),(\gamma,\phi) \in X\times Y.
\end{align}
Thereby,  (\ref{desigualdade4np}) and (\ref{desigualdade4n1p}) show that the operator $B$ is continuous and, moreover, one obtains
\begin{align}\label{estimativaBp}
	\|B((w,v),(\gamma,\phi))\|_{X\times Y}
&\leq C_{a,\sigma,s}[T^{1-\frac{1}{2\alpha}}+T^{1-\frac{1}{2\beta}}]\|(w,v) \|_{X\times Y}\|(\gamma,\phi)\|_{X\times Y},\quad \forall (w,v),(\gamma,\phi) \in X\times Y.
\end{align}

Assume that $\rho\in[0,1)$ is such that
\begin{align}\label{rhop}
\rho=\min\Big\{1-\frac{1}{2\alpha},1-\frac{1}{2\beta}\Big\},
\end{align}
provided $\alpha\geq \frac{1}{2}$ and $\beta\geq \frac{1}{2}$, in order to split this proof into two cases.\\\\
\bigskip
\noindent \underline{Case 1}: Consider that 
\begin{align}\label{p13}
\alpha>\frac{1}{2},\,\,\beta>\frac{1}{2}\hbox{  and  } 0<T< [2+(16C_{a,\sigma,s}\|(u_0,\theta_0)\|_{\mathcal{X}_{a,\sigma}^{s}})^{\frac{1}{2}}]^{-\frac{2}{\rho}},
\end{align}
 where $C_{a,\sigma,s}$ is given in (\ref{estimativaBp}).

In this case, we have $0<T<\frac{1}{2}$, and $0<\rho<1$ (see (\ref{rhop})). As a result, (\ref{desigualdade1n}) and (\ref{p13}) imply that
\begin{align}\label{p15}
	\nonumber &4C_{a,\sigma,s}[T^{1-\frac{1}{2\alpha}}+T^{1-\frac{1}{2\beta}}]\|(e^{- t(-\Delta)^{\alpha}}u_0,e^{- t(-\Delta)^{\beta}}\theta_0)\|_{X\times Y}\\
&\leq 16C_{a,\sigma,s}T^\rho\|(u_0,\theta_0)\|_{\mathcal{X}_{a,\sigma}^{s}}<(1-2T)^2.
\end{align}

\bigskip
\noindent \underline{Case 2}: Assume that 
\begin{align}\label{p14}
\alpha=\frac{1}{2} \hbox{  or  } \beta=\frac{1}{2},\,\, \|(u_0,\theta_0)\|_{\mathcal{X}_{a,\sigma}^{s}}<[16C_{a,\sigma,s}]^{-1} \hbox{  and  } 0<T< \frac{1}{2}-2[C_{a,\sigma,s}\|(u_0,\theta_0)\|_{\mathcal{X}_{a,\sigma}^{s}}]^{\frac{1}{2}}, 
\end{align}
where $C_{a,\sigma,s}$ is given in (\ref{estimativaBp}).

Note that  $0<T<\frac{1}{2}$ and $\rho=0$ (see (\ref{rhop})). Thereby, by (\ref{desigualdade1n}) and (\ref{p14}), we infer
\begin{align}\label{p16}
	4C_{a,\sigma,s}[T^{1-\frac{1}{2\alpha}}+T^{1-\frac{1}{2\beta}}]\|(e^{- t(-\Delta)^{\alpha}}u_0,e^{- t(-\Delta)^{\beta}}\theta_0)\|_{X\times Y}
\leq 16C_{a,\sigma,s}\|(u_0,\theta_0)\|_{\mathcal{X}_{a,\sigma}^{s}}<(1-2T)^2.
\end{align}

By using these two cases above, (\ref{estimativaprojecao3}), (\ref{lwfinal}), (\ref{estimativaBp}), (\ref{p15}), (\ref{p16}) and Lemma \ref{lemaB}, there is a unique $(u,\theta)\in X\times Y$ such that
$$\|(u,\theta)\|_{X\times Y}< \frac{1-2T}{2C_{a,\sigma,s}[T^{1-\frac{1}{2\alpha}}+T^{1-\frac{1}{2\beta}}]},$$
and, moreover, 
\begin{align*}
\|(u,\theta)\|_{X\times Y}\leq \frac{2\|(e^{- t(-\Delta)^{\alpha}}u_0,e^{- t(-\Delta)^{\beta}}\theta_0)\|_{X\times Y}}{1-2T}\leq \frac{4\|(u_0,\theta_0)\|_{\mathcal{X}_{a,\sigma}^{s}}}{1-2T}.
\end{align*}

Lastly, let us reapply Lemma \ref{lemacalor} to the system (\ref{sistemauteta}). Thus, by (\ref{p4}), (\ref{existencianova2n1p}) and (\ref{existencianova2np}), one infers
 \begin{align}\label{p17}
	\|\mathbb{P}[u\cdot\nabla u]\|_{L^1_T(\mathcal{X}_{a,\sigma}^{s})}&\leq C_{a,\sigma,s}T^{1-\frac{1}{2\alpha}}\|u \|_{X}^2<\infty
\end{align}
and also
\begin{align}\label{p18}
	\|u\cdot\nabla \theta\|_{L^1_T(\mathcal{X}_{a,\sigma}^{s})}&\leq C_{a,\sigma,s}[T^{1-\frac{1}{2\alpha}}+T^{1-\frac{1}{2\beta}}]\|u \|_{X}\|\theta\|_{Y}<\infty,
\end{align}
since $(u,\theta)\in X\times Y.$ Therefore, by applying (\ref{sistemauteta}), (\ref{p4}), (\ref{p17}), (\ref{p18}) and Lemma \ref{lemacalor}, one deduces
$$(u,\theta)\in L^p_{T}(\mathcal{X}_{a,\sigma}^{s+\frac{2\alpha}{p}}(\mathbb{R}^3))\times L^p_{T}(\mathcal{X}_{a,\sigma}^{s+\frac{2\beta}{p}}(\mathbb{R}^3)),\quad \forall p\geq1.$$
\caixa

\section{Blow-up criteria for local solutions}\label{secaoblowupB}

In this section, we shall present the proofs of Theorems \ref{teoremaB2} and \ref{teoremaB1}, and Corollaries \ref{corollaryB1} and \ref{corollaryB2}. These four results are focused on showing the behavior at potential blow-up times for the mild solutions  of the Boussinesq equations (\ref{MHD-alpha}) obtained in Theorems \ref{teoremaexistenciaB2} and \ref{teoremaexistenciaB}.


\bigskip

\noindent\textbf{Proof of Theorems \ref{teoremaB2} i) and \ref{teoremaB1} i):} Suppose that Theorems \ref{teoremaB2} i) and \ref{teoremaB1} \textbf{i)} are not true. This means that
\begin{align}\label{ls1}
\displaystyle \limsup_{t\nearrow T^*} \|(u,\theta)(t)\|_{\mathcal{X}_{a,\sigma}^{s}}<\infty.
\end{align}
Our goal is to verify  that our solution $(u,\theta)(\cdot, t)$ can be extended beyond $t=T^*$. Indeed,  (\ref{ls1}) and Theorems  \ref{teoremaexistenciaB2} and \ref{teoremaexistenciaB} imply that
\begin{align}\label{ls2}
\|(u,\theta)(t)\|_{\mathcal{X}_{a,\sigma}^{s}} \leq C_1, \quad \forall t\in [0,T^*),
\end{align}
where $C_1=C_{a,\sigma,s,T^*}$ is a positive constant. On the other hand, by taking Fourier Transform and the scalar product in $\mathbb{C}^3$ of the first equation of ($\ref{MHD-alpha}$) with $\widehat{u}(t)$, one has
\begin{align*}
\frac{1}{2}\partial_t|\widehat{u}(t)|^2+|\xi|^{2\alpha}|\hat{u}|^2 &\leq
|\widehat{u}\cdot \widehat{u\cdot\nabla \displaystyle u}|+|\widehat{u}_3\widehat{\theta}|.
\end{align*}
Thus, by choosing $\varepsilon>0$, one infers
\begin{align*}
 \partial_{t}\sqrt{|\widehat{u}(t)|^2+\varepsilon}+\frac{|\xi|^{2\alpha}|\widehat{u}|^2}{\sqrt{|\widehat{u}|^2+\varepsilon}}
&\leq|\widehat{u\cdot\nabla \displaystyle u}|+|\widehat{\theta}|.
\end{align*}
Integrate the result above over $[0,t]$ in order to obtain
\begin{align*} \sqrt{|\widehat{u}(t)|^2+\varepsilon}+ \int_{0}^{t}\frac{|\xi|^{2\alpha}|\widehat{u}(\tau)|^2}{\sqrt{|\widehat{u}(\tau)|^2+\varepsilon}}\,d\tau\leq\sqrt{|\widehat{u}_0|^2+\varepsilon}
+\int_{0}^{t}[|\widehat{(u\cdot\nabla \displaystyle u)}(\tau)|+|\widehat{\theta}(\tau)|]\,d\tau.
\end{align*}
By passing to the limit the inequality above, as $\varepsilon\searrow0$, multiplying by $|\xi|^{s}e^{a|\xi|^{\frac{1}{\sigma}}}$ and integrating over $\xi \in\mathbb{R}^3$, one can write
\begin{align}\label{ls3}
\|u(t)\|_{\mathcal{X}_{a,\sigma}^{s}} + \int_{0}^{t}\|u(\tau)\|_{\mathcal{X}_{a,\sigma}^{s+2\alpha}}\;d\tau\leq \|u_0\|_{\mathcal{X}_{a,\sigma}^{s}}+\int_{0}^{t}[\|u\cdot\nabla u(\tau)\|_{\mathcal{X}_{a,\sigma}^{s}} + \|\theta(\tau)\|_{\mathcal{X}_{a,\sigma}^{s}}]\;d\tau,
\end{align}
for all $t\in [0,T^*)$. Analogously, one reaches the following inequality: 
\begin{align}\label{ls3b}
\|\theta(t)\|_{\mathcal{X}_{a,\sigma}^{s}} + \int_{0}^{t}\|\theta(\tau)\|_{\mathcal{X}_{a,\sigma}^{s+2\beta}}\;d\tau\leq \|\theta_0\|_{\mathcal{X}_{a,\sigma}^{s}}+\int_{0}^{t}\|u\cdot\nabla \theta(\tau)\|_{\mathcal{X}_{a,\sigma}^{s}}\;d\tau,\quad \forall t\in [0,T^*).
\end{align}

On the other hand, let us denote  $(\kappa_n)_{n\in\mathbb{N}}$ as being an arbitrary sequence
such that $0<\kappa_n<T^*$ and
$\kappa_n\nearrow T^*$, as $n\rightarrow\infty$.  Thereby, we shall prove that
\begin{align}\label{e5}
\lim_{n,m\rightarrow \infty} \|(u,\theta)(\kappa_n)-(u,\theta)(\kappa_m)\|_{\mathcal{X}_{a,\sigma}^s}=0.
\end{align}
As follows, it will be assumed that $\kappa_m\leq \kappa_n$ (without loss of generality). By  (\ref{estimativaprojecao3}), we can write the next equality:
\begin{align}\label{wp6}
(u,\theta)(\kappa_n)-(u,\theta)(\kappa_m)= I_1+I_2+I_3,
\end{align}
where
\begin{align}\label{I1}
I_1&= (I_{11},I_{12}):= ([e^{-  \kappa_n (-\Delta)^\alpha}-e^{-\kappa_m(-\Delta)^\alpha}]u_0,[e^{-\kappa_n (-\Delta)^\beta}-e^{- \kappa_m (-\Delta)^\beta}]\theta_0),
\end{align}
\begin{align}\label{wilber19}
\nonumber &I_2=(I_{21}^{(1)}+I_{21}^{(2)},I_{22}):= \Big(\int_{0}^{\kappa_m}[e^{- (\kappa_m-\tau)(-\Delta)^\alpha}-e^{- (\kappa_n-\tau)(-\Delta)^\alpha}] \mathbb{P}[u\cdot \nabla u]\,d\tau+\\ &
\int_{0}^{\kappa_m}[e^{-  (\kappa_n-\tau)(-\Delta)^\alpha}-e^{- (\kappa_m-\tau)(-\Delta)^\alpha}] \mathbb{P}[\theta e_3]\,d\tau,\int_{0}^{\kappa_m}[e^{- (\kappa_m-\tau)(-\Delta)^\beta}-e^{- (\kappa_n-\tau)(-\Delta)^\beta}] (u\cdot \nabla \theta)\,d\tau\Big)
\end{align}
and also
\begin{align}\label{wilber20}
\nonumber I_3=(I_{31}^{(1)}+I_{31}^{(2)},I_{32})&:= \Big(-\int_{\kappa_m}^{\kappa_n} e^{- (\kappa_n-\tau)(-\Delta)^\alpha} \mathbb{P}[u\cdot \nabla u]\,d\tau+
\int_{\kappa_m}^{\kappa_n} e^{- (\kappa_n-\tau)(-\Delta)^\alpha} \mathbb{P}[\theta e_3]\,d\tau,\\
&\quad-\int_{\kappa_m}^{\kappa_n}e^{- (\kappa_n-\tau)(-\Delta)^\beta} (u\cdot \nabla \theta)\,d\tau\Big).
\end{align}

In order to verify (\ref{e5}), let us split the proof of this statement into two cases.

\bigskip
\noindent \underline{Case 1}: Assume that 
$$a\geq0,\,\, \sigma \geq1,\,\,  (u_0,\theta_0)\in \mathcal{X}_{a,\sigma}^s(\mathbb{R}^3)$$
and also that
$$s> \max\{1-2\alpha,\tfrac{\alpha(1-2\beta)}{\beta},\tfrac{\beta(1-2\alpha)}{\alpha}\}\,\, \hbox{  and  } \,\, s\in [-1,0].$$

First of all, notice that the conditions above for $s$ show that $\alpha>\frac{1}{2}$ and $\beta>\frac{1}{2}$.  On the other hand, by (\ref{wp3}), (\ref{wp1}), (\ref{ls2}), (\ref{ls3}) and (\ref{ls3b}), one obtains
\begin{align}\label{ls4}
\nonumber\|u(t)\|_{\mathcal{X}_{a,\sigma}^{s}} + \int_{0}^{t}\|u(\tau)\|_{\mathcal{X}_{a,\sigma}^{s+2\alpha}}\;d\tau&\leq \|u_0\|_{\mathcal{X}_{a,\sigma}^{s}}
+C_{s}C_1^{2+\frac{s-1}{2\alpha}}\int_0^t\|u\|_{\mathcal{X}^{s+2\alpha}_{a,\sigma}}^{\frac{1-s}{2\alpha}}\,dt+\int_{0}^{t} \|\theta\|_{\mathcal{X}_{a,\sigma}^{s}}\;d\tau\\
&\leq \|u_0\|_{\mathcal{X}_{a,\sigma}^{s}}
+C_{s}C_1^{2+\frac{s-1}{2\alpha}}\int_0^t\|u\|_{\mathcal{X}^{s+2\alpha}_{a,\sigma}}^{\frac{1-s}{2\alpha}}\,dt+C_1T^*
\end{align}
and also
\begin{align}\label{ls4b}
\nonumber\|\theta(t)\|_{\mathcal{X}_{a,\sigma}^{s}} + \int_{0}^{t}\|\theta(\tau)\|_{\mathcal{X}_{a,\sigma}^{s+2\beta}}\;d\tau&\leq \|\theta_0\|_{\mathcal{X}_{a,\sigma}^{s}}+C_{s}C_1^{2+\frac{s}{2\beta}-\frac{1}{2\alpha}}\int_0^t\|\theta\|_{\mathcal{X}^{s+2\beta}_{a,\sigma}}^{\frac{-s}{2\beta}}\|u\|_{\mathcal{X}^{s+2\alpha}_{a,\sigma}}^{\frac{1}{2\alpha}}\,dt\\
&\quad+C_{s}C_1^{2+\frac{s}{2\alpha}-\frac{1}{2\beta}}\int_0^t \|u\|_{\mathcal{X}^{s+2\alpha}_{a,\sigma}}^{\frac{-s}{2\alpha}} \|\theta\|_{\mathcal{X}^{s+2\beta}_{a,\sigma}}^{\frac{1}{2\beta}}\,dt,
\end{align}
for all $t\in[0,T^*)$, provided that $a\geq0$, $\sigma\geq1$, $\max\{-1,-2\alpha, -2\beta\}\leq s \leq0$, $\alpha\geq\frac{1}{2}$ and $\beta\geq\frac{1}{2}$. Thus,
by using Young's inequality, one infers
\begin{align}\label{w1}
&\|u(t)\|_{\mathcal{X}_{a,\sigma}^{s}} + \int_{0}^{t}\|u(\tau)\|_{\mathcal{X}_{a,\sigma}^{s+2\alpha}}\;d\tau\leq \|u_0\|_{\mathcal{X}_{a,\sigma}^{s}}
+\frac{1}{4}\int_{0}^{t}\|u(\tau)\|_{\mathcal{X}_{a,\sigma}^{s+2\alpha}}\;d\tau+ C_{a,\sigma,s,\alpha,T^*},
\end{align}
provided that $s>1-2\alpha$, and also
\begin{align}\label{w2}
&\|\theta(t)\|_{\mathcal{X}_{a,\sigma}^{s}} + \int_{0}^{t}\|\theta(\tau)\|_{\mathcal{X}_{a,\sigma}^{s+2\beta}}\;d\tau\leq \|\theta_0\|_{\mathcal{X}_{a,\sigma}^{s}}
+ \frac{1}{4}\int_{0}^{t}\|u(\tau)\|_{\mathcal{X}_{a,\sigma}^{s+2\alpha}}\;d\tau +
\frac{1}{2}\int_{0}^{t}\|\theta(\tau)\|_{\mathcal{X}_{a,\sigma}^{s+2\beta}}\;
d\tau
+ C_{a,\sigma,s,\alpha,\beta,T^*},
\end{align}
provided that $s>\max\{\tfrac{\alpha(1-2\beta)}{\beta},\tfrac{\beta(1-2\alpha)}{\alpha}\}$. Then, by adding (\ref{w1}) to (\ref{w2}), one deduces
\begin{align*}
&\|(u,\theta)(t)\|_{\mathcal{X}_{a,\sigma}^{s}}+\frac{1}{2}\int_{0}^{t}\|u(\tau)\|_{\mathcal{X}_{a,\sigma}^{s+2\alpha}}\;d\tau+\frac{1}{2}\int_{0}^{t}\|\theta(\tau)\|_{\mathcal{X}_{a,\sigma}^{s+2\beta}}\;d\tau\leq \|(u_0,\theta_0)\|_{\mathcal{X}_{a,\sigma}^{s}}
+ C_{a,\sigma,s,\alpha,\beta,T^*}.
\end{align*}
As a result, one has
\begin{align}\label{ls5}
\int_{0}^{t}\|u(\tau)\|_{\mathcal{X}_{a,\sigma}^{s+2\alpha}}\;d\tau +
\int_{0}^{t}\|\theta(\tau)\|_{\mathcal{X}_{a,\sigma}^{s+2\beta}}\;
d\tau  &\leq C_2,\quad\forall t\in [0,T^*),
\end{align}
where $C_2=C_{a,\sigma,s,\alpha,\beta,u_0,\theta_0,T^*}$ is a positive constant. 

Now, we are ready to prove (\ref{e5}) in this first case. At first, by (\ref{I1}), it holds the following inequality:
\begin{align*}
\|I_{12}\|_{\mathcal{X}_{a,\sigma}^{s}}
&\leq \int_{\mathbb{R}^3} |\xi|^{s}e^{a|\xi|^{\frac{1}{\sigma}}}(e^{-\kappa_m|\xi|^{2\beta}}-e^{-T^*|\xi|^{2\beta}})|\widehat{\theta}_0(\xi)|\;d\xi.
\end{align*}
By applying the fact that  $\theta_0\in\mathcal{X}_{a,\sigma}^s(\mathbb{R}^3)$, we infer
$\displaystyle 
\lim_{n,m\to\infty}\|I_{12}\|_{\mathcal{X}_{a,\sigma}^{s}}=0.
$
Similarly, we can claim
$
\displaystyle \lim_{n,m\to\infty}\|I_{11}\|_{\mathcal{X}_{a,\sigma}^{s}}=0,
$
since $u_0\in\mathcal{X}_{a,\sigma}^s(\mathbb{R}^3)$. Consequently, one reaches
$\displaystyle 
\lim_{n,m\to\infty}\|I_1\|_{\mathcal{X}_{a,\sigma}^{s}}=0.$

On the other hand, by (\ref{wilber19}), one can write
\begin{align*}
\|I_{21}^{(2)}\|_{\mathcal{X}_{a,\sigma}^{s}}&\leq \int_{0}^{\kappa_m}\|[e^{-  (\kappa_n-\tau)(-\Delta)^\alpha}-e^{- (\kappa_m-\tau)(-\Delta)^\alpha}] \mathbb{P}[\theta e_3]\|_{\mathcal{X}_{a,\sigma}^{s}}\,d\tau\\
&\leq \int_{0}^{\kappa_m}\int_{\mathbb{R}^3}|\xi|^se^{a|\xi|^{\frac{1}{\sigma}}}[e^{-(\kappa_m-\tau)|\xi|^{2\alpha}}-e^{-(\kappa_n-\tau)|\xi|^{2\alpha}}]|\widehat{\theta}(\tau)|\;d\xi d\tau\\
&\leq \int_{0}^{T^*}\int_{\mathbb{R}^3}|\xi|^se^{a|\xi|^{\frac{1}{\sigma}}}[1-e^{-(T^*-\kappa_m)|\xi|^{2\alpha}}]|\widehat{\theta}(\tau)|\;d\xi d\tau.
\end{align*}
As a result, $\displaystyle\lim_{n,m\rightarrow\infty}\|I_{21}^{(2)}\|_{\mathcal{X}_{a,\sigma}^{s}}=0$; since, by (\ref{ls2}), it holds
\begin{align*}
 \int_{0}^{T^*}\int_{\mathbb{R}^3}|\xi|^se^{a|\xi|^{\frac{1}{\sigma}}}|\widehat{\theta}(\tau)|\;d\xi d\tau= \int_{0}^{T^*}\|\theta(\tau)\|_{\mathcal{X}_{a,\sigma}^{s}}\; d\tau\leq C_1T^*<\infty.
\end{align*}
Moreover, it follows that
\begin{align}\label{wt4}
\nonumber\|I_{22}\|_{\mathcal{X}_{a,\sigma}^{s}}	&\leq
	\int_{0}^{\kappa_m}\int_{\mathbb{R}^3}|\xi|^{s}e^{a|\xi|^{\frac{1}{\sigma}}}[e^{-(\kappa_m-\tau)|\xi|^{2\beta}}-e^{-(\kappa_n-\tau)|\xi|^{2\beta}}]|\mathcal{F}(u\cdot\nabla \theta)(\xi)|\;d\xi\, d\tau\\
	&\qquad\leq
	\int_{0}^{T^*}\int_{\mathbb{R}^3}|\xi|^{s}e^{a|\xi|^{\frac{1}{\sigma}}}[1-e^{-(T^*-\kappa_m)|\xi|^{2\beta}}]|\mathcal{F}(u\cdot\nabla\theta)(\xi)|\;d\xi\, d\tau.
	\end{align}
As a consequence, $\displaystyle \lim_{n,m\rightarrow \infty}\|I_{22}\|_{\mathcal{X}_{a,\sigma}^{s}}=0$; since, by  (\ref{wp4}), (\ref{ls2}) and (\ref{ls5}), we obtain
\begin{align*}
\nonumber&\int_{0}^{T^*}\|u\cdot\nabla \theta\|_{\mathcal{X}_{a,\sigma}^s}\; d\tau\leq C_{s}[(T^{*})^{1+\frac{s}{2\beta}-\frac{1}{2\alpha}}C_1^{2+\frac{s}{2\beta}-\frac{1}{2\alpha}}C_{2}^{\frac{1}{2\alpha}-\frac{s}{2\beta}} + (T^{*})^{1+\frac{s}{2\alpha}-\frac{1}{2\beta}}C_1^{2+\frac{s}{2\alpha}-\frac{1}{2\beta}}C_{2}^{\frac{1}{2\beta}-\frac{s}{2\alpha}}]<\infty.
\end{align*}
(Here, $a\geq0$, $\sigma\geq1$, $\max\{-1,\frac{\beta(1-2\alpha)}{\alpha},\frac{\alpha(1-2\beta)}{\beta}\}\leq s \leq0$, $\alpha\geq\frac{1}{2}$ and $\beta\geq\frac{1}{2}$).
Similarly, by  (\ref{wp5}), (\ref{ls2}) and (\ref{ls5}), we deduce
$
\displaystyle \lim_{n,m\to\infty}\|I_{21}^{(1)}\|_{\mathcal{X}_{a,\sigma}^{s}}=0,$ provided that $a\geq0$, $\sigma\geq1$, $\max\{-1,1-2\alpha\}\leq s \leq0$ and $\alpha\geq\frac{1}{2}$. Thereby,
$\displaystyle \lim_{n,m\to\infty}\|I_2\|_{\mathcal{X}_{a,\sigma}^{s}}=0.$

Lastly, by (\ref{wilber20}) and (\ref{ls2}), one has
\begin{align}\label{wt5}
\nonumber\|I_{31}^{(2)}\|_{\mathcal{X}_{a,\sigma}^{s}}&\leq \int_{\kappa_m}^{\kappa_n}\int_{\mathbb{R}^3}|\xi|^se^{a|\xi|^{\frac{1}{\sigma}}}e^{-(\kappa_n-\tau)|\xi|^{2\alpha}}|\mathcal{F}[\mathbb{P}(\theta e_3)]|\;d\xi d\tau\\
&\leq \int_{\kappa_m}^{T^*}\|\theta(\tau)\|_{\mathcal{X}_{a,\sigma}^{s}} d\tau\leq C_1(T^*-\kappa_m).
\end{align}
Then, $\displaystyle \lim_{n,m\to\infty}\|I_{31}^{(2)}\|_{\mathcal{X}_{a,\sigma}^{s}}=0.$ Furthermore, analogously to (\ref{wp4}), and also by (\ref{ls2}) and (\ref{ls5}), we  obtain
\begin{align*}
\|I_{32}\|_{\mathcal{X}_{a,\sigma}^{s}}
\leq C_{s}[(T^*-\kappa_m)^{1+\frac{s}{2\beta}-\frac{1}{2\alpha}}C_1^{2+\frac{s}{2\beta}-\frac{1}{2\alpha}}C_2^{\frac{1}{2\alpha}-\frac{s}{2\beta}}+(T^*-\kappa_m)^{1+\frac{s}{2\alpha}-\frac{1}{2\beta}}
C_1^{2+\frac{s}{2\alpha}-\frac{1}{2\beta}} C_2^{\frac{1}{2\beta}-\frac{s}{2\alpha}}],
	\end{align*}
provided that $a\geq0$, $\sigma\geq1$, $\max\{-1,\frac{\beta(1-2\alpha)}{\alpha},\frac{\alpha(1-2\beta)}{\beta}\}\leq s \leq0$, $\alpha\geq\frac{1}{2}$ and $\beta\geq\frac{1}{2}$. As a consequence, we can write $\displaystyle \lim_{n,m\to\infty}\|I_{32}\|_{\mathcal{X}_{a,\sigma}^{s}}=0$ if $s>\max\{\frac{\beta(1-2\alpha)}{\alpha},\frac{\alpha(1-2\beta)}{\beta}\} $, $\alpha>\frac{1}{2}$ and $\beta>\frac{1}{2}$. Analogously, from (\ref{wp5}), (\ref{ls2}) and (\ref{ls5}), we can write
\begin{align*}
\|I_{32}^{(1)}\|_{\mathcal{X}_{a,\sigma}^{s}}
\leq C_{s}(T^*-\kappa_m)^{1+\frac{s-1}{2\alpha}}C_1^{2+\frac{s-1}{2\alpha}}C_2^{\frac{1-s}{2\alpha}},
	\end{align*}
provided that $a\geq0$, $\sigma\geq1$, $s\in[-1,0]$, $1-2\alpha< s$ and $\alpha>\frac{1}{2}$. By passing to the limit, as $n,m\to\infty$, one infers
$\displaystyle 
\lim_{n,m\to\infty}\|I_{32}^{(1)}\|_{\mathcal{X}_{a,\sigma}^{s}}=0$. These arguments lead us to $\displaystyle 
\lim_{n,m\to\infty}\|I_{3}\|_{\mathcal{X}_{a,\sigma}^{s}}=0$.  This completes the proof of (\ref{e5}) (see also (\ref{wp6})).

\bigskip
\noindent \underline{Case 2}: Assume that $$\alpha>\frac{1}{2}, \,\,\beta>\frac{1}{2},\,\,(u_0,\theta_0)\in \mathcal{X}_{a,\sigma}^s(\mathbb{R}^3)$$
and also
$$(a,s,\sigma)\in\{(0,\infty)\times [-1,0)\times (1,\infty)\}\cup \{[0,\infty)\times\{0\}\times [1,\infty)\}.$$  

At first, by (\ref{wp2}), (\ref{wt2}), (\ref{ls2}), (\ref{ls3}) and (\ref{ls3b}), it is true that
\begin{align}\label{ls4wp}
\nonumber\|u(t)\|_{\mathcal{X}_{a,\sigma}^{s}} + \int_{0}^{t}\|u(\tau)\|_{\mathcal{X}_{a,\sigma}^{s+2\alpha}}\;d\tau&\leq \|u_0\|_{\mathcal{X}_{a,\sigma}^{s}}
+C_{a,\sigma,s}C_1^{2-\frac{1}{2\alpha}}\int_0^t\|u\|_{\mathcal{X}^{s+2\alpha}_{a,\sigma}}^{\frac{1}{2\alpha}}\,dt+\int_{0}^{t} \|\theta\|_{\mathcal{X}_{a,\sigma}^{s}}\;d\tau\\
&\leq \|u_0\|_{\mathcal{X}_{a,\sigma}^{s}}
+C_{a,\sigma,s}C_1^{2-\frac{1}{2\alpha}}\int_0^t\|u\|_{\mathcal{X}^{s+2\alpha}_{a,\sigma}}^{\frac{1}{2\alpha}}\,dt+C_1T^*
\end{align}
 and also
\begin{align}\label{ls4bwp}
\|\theta(t)\|_{\mathcal{X}_{a,\sigma}^{s}} + \int_{0}^{t}\|\theta(\tau)\|_{\mathcal{X}_{a,\sigma}^{s+2\beta}}\;d\tau&\leq \|\theta_0\|_{\mathcal{X}_{a,\sigma}^{s}}+C_{a,\sigma,s}C_1^{2-\frac{1}{2\alpha}}\int_0^t \|u\|_{\mathcal{X}^{s+2\alpha}_{a,\sigma}}^{\frac{1}{2\alpha}}\,dt+C_{a,\sigma,s}C_1^{2-\frac{1}{2\beta}}\int_0^t  \|\theta\|_{\mathcal{X}^{s+2\beta}_{a,\sigma}}^{\frac{1}{2\beta}}\,dt,
\end{align}
for all $t\in [0,T^*)$, provided that $(a,s,\sigma)\in \{(0,\infty)$ $\times[-1,0)\times (1,\infty) \}\cup \{[0,\infty)\times\{0\}\times [1,\infty) \}$, $\alpha\geq\frac{1}{2}$ and $\beta\geq\frac{1}{2}$. Thus,
by using Young's inequality and  (\ref{ls2}), one infers
\begin{align}\label{w1wp}
&\|u(t)\|_{\mathcal{X}_{a,\sigma}^{s}} + \int_{0}^{t}\|u(\tau)\|_{\mathcal{X}_{a,\sigma}^{s+2\alpha}}\;d\tau\leq \|u_0\|_{\mathcal{X}_{a,\sigma}^{s}}
+\frac{1}{4}\int_{0}^{t}\|u(\tau)\|_{\mathcal{X}_{a,\sigma}^{s+2\alpha}}\;d\tau+ C_{a,\sigma,s,\alpha,T^*}
\end{align}
and also
\begin{align}\label{w2wp}
&\|\theta(t)\|_{\mathcal{X}_{a,\sigma}^{s}} + \int_{0}^{t}\|\theta(\tau)\|_{\mathcal{X}_{a,\sigma}^{s+2\beta}}\;d\tau\leq \|\theta_0\|_{\mathcal{X}_{a,\sigma}^{s}}
+ \frac{1}{4}\int_{0}^{t}\|u(\tau)\|_{\mathcal{X}_{a,\sigma}^{s+2\alpha}}\;d\tau +
\frac{1}{2}\int_{0}^{t}\|\theta(\tau)\|_{\mathcal{X}_{a,\sigma}^{s+2\beta}}\;
d\tau
+ C_{a,\sigma,s,\alpha,\beta,T^*},
\end{align}
provided that $\alpha>\frac{1}{2}$ and $\beta>\frac{1}{2}$. Therefore, by observing (\ref{w1wp}) to (\ref{w2wp}), we conclude
\begin{align*}
&\|(u,\theta)(t)\|_{\mathcal{X}_{a,\sigma}^{s}}+\frac{1}{2}\int_{0}^{t}\|u(\tau)\|_{\mathcal{X}_{a,\sigma}^{s+2\alpha}}\;d\tau+\frac{1}{2}\int_{0}^{t}\|\theta(\tau)\|_{\mathcal{X}_{a,\sigma}^{s+2\beta}}\;d\tau\leq \|(u_0,\theta_0)\|_{\mathcal{X}_{a,\sigma}^{s}}
+ C_{a,\sigma,s,\alpha,\beta,T^*},
\end{align*}
for all $t\in[0,T^*)$. This means that
\begin{align}\label{ls5wp}
\int_{0}^{t}\|u(\tau)\|_{\mathcal{X}_{a,\sigma}^{s+2\alpha}}\;d\tau +
\int_{0}^{t}\|\theta(\tau)\|_{\mathcal{X}_{a,\sigma}^{s+2\beta}}\;
d\tau  &\leq C_3,\quad\forall t\in [0,T^*),
\end{align}
where $C_3=C_{a,\sigma,s,\alpha,\beta,u_0,\theta_0,T^*}$ is a positive constant. 

Now, we are ready to prove (\ref{e5}) in the second case. In fact, notice that
$\displaystyle 
\lim_{n,m\to\infty}\|I_1\|_{\mathcal{X}_{a,\sigma}^{s}}=0$ as in the first case.

It is also true that $\displaystyle\lim_{n,m\rightarrow\infty}\|I_{21}^{(2)}\|_{\mathcal{X}_{a,\sigma}^{s}}=0$ as in the first case.
In addition, by  (\ref{wt3}), (\ref{ls2}) and (\ref{ls5wp}), we conclude
\begin{align*}
\nonumber&\int_{0}^{T^*}\|u\cdot\nabla \theta\|_{\mathcal{X}_{a,\sigma}^s}\; d\tau\leq C_{a,\sigma,s}[(T^{*})^{1-\frac{1}{2\alpha}}C_1^{2-\frac{1}{2\alpha}}C_{3}^{\frac{1}{2\alpha}} + (T^{*})^{1-\frac{1}{2\beta}}C_1^{2-\frac{1}{2\beta}}C_{3}^{\frac{1}{2\beta}}]<\infty,
\end{align*}
provided that $(a,s,\sigma)\in \{(0,\infty)$ $\times[-1,0)\times (1,\infty) \}\cup \{[0,\infty)\times\{0\}\times [1,\infty) \}$, $\alpha\geq\frac{1}{2}$ and $\beta\geq\frac{1}{2}$. Thereby, by applying (\ref{wt4}), we deduce that $
\displaystyle \lim_{n,m\to\infty}\|I_{22}\|_{\mathcal{X}_{a,\sigma}^{s}}=0$. Similarly, by  (\ref{wt1}), (\ref{ls2}) and (\ref{ls5wp}), 
$\displaystyle \lim_{n,m\to\infty}\|I_{21}^{(1)}\|_{\mathcal{X}_{a,\sigma}^{s}}=0,$ provided also that $(a,s,\sigma)\in \{(0,\infty)$ $\times[-1,0)\times (1,\infty) \}\cup \{[0,\infty)\times\{0\}\times [1,\infty) \}$, $\alpha\geq\frac{1}{2}$ and $\beta\geq\frac{1}{2}$. Therefore, we reach the limit
$\displaystyle \lim_{n,m\to\infty}\|I_2\|_{\mathcal{X}_{a,\sigma}^{s}}=0.$

Lastly, by (\ref{wilber20}) and (\ref{wt5}), we conclude that
 $\displaystyle \lim_{n,m\to\infty}\|I_{31}^{(2)}\|_{\mathcal{X}_{a,\sigma}^{s}}=0.$ Furthermore, analogously to (\ref{wt3}), and also by (\ref{ls2}) and (\ref{ls5wp}), we  obtain
\begin{align*}
\|I_{32}\|_{\mathcal{X}_{a,\sigma}^{s}}
\leq C_{a,\sigma,s}[(T^*-\kappa_m)^{1-\frac{1}{2\alpha}}C_1^{2-\frac{1}{2\alpha}}C_3^{\frac{1}{2\alpha}}+(T^*-\kappa_m)^{1-\frac{1}{2\beta}}
C_1^{2-\frac{1}{2\beta}} C_3^{\frac{1}{2\beta}}],
	\end{align*}
provided that $(a,s,\sigma)\in \{(0,\infty)$ $\times[-1,0)\times (1,\infty) \}\cup \{[0,\infty)\times\{0\}\times [1,\infty) \}$, $\alpha\geq\frac{1}{2}$ and $\beta\geq\frac{1}{2}$. Hence,  $\displaystyle \lim_{n,m\to\infty}\|I_{32}\|_{\mathcal{X}_{a,\sigma}^{s}}=0$ if $\alpha>\frac{1}{2}$ and $\beta>\frac{1}{2}$. Similarly, (\ref{wt1}), (\ref{ls2}) and (\ref{ls5wp}), we have
\begin{align*}
\|I_{32}^{(1)}\|_{\mathcal{X}_{a,\sigma}^{s}}
\leq C_{a,\sigma,s}(T^*-\kappa_m)^{1-\frac{1}{2\alpha}}C_1^{2-\frac{1}{2\alpha}}C_3^{\frac{1}{2\alpha}},
	\end{align*}
provided that $(a,s,\sigma)\in \{(0,\infty)$ $\times[-1,0)\times (1,\infty) \}\cup \{[0,\infty)\times\{0\}\times [1,\infty) \}$ and $\alpha\geq\frac{1}{2}$. Hence,
$\displaystyle 
\lim_{n,m\to\infty}\|I_{32}^{(1)}\|_{\mathcal{X}_{a,\sigma}^{s}}=0$ if $\alpha>\frac{1}{2}$. Therefore, $\displaystyle 
\lim_{n,m\to\infty}\|I_{3}\|_{\mathcal{X}_{a,\sigma}^{s}}=0$.  This completes the proof of (\ref{e5}) (see also (\ref{wp6})).

The limit (\ref{e5}) lets us know that $((u,\theta)(\kappa_n))_{n\in\mathbb{N}}$ is a Cauchy sequence in  Banach space $\mathcal{X}_{a,\sigma}^s(\mathbb{R}^3)$ (recall that $s\leq0$).  This means that there exists $(u^{(1)},\theta^{(1)})\in\mathcal{X}_{a,\sigma}^s(\mathbb{R}^3)$ such that
\begin{align}\label{wt6}
\lim_{n\rightarrow \infty} \|(u,\theta)(\kappa_n)- (u^{(1)},\theta^{(1)})\|_{\mathcal{X}_{a,\sigma}^{s}}=0.
\end{align}
We claim that the limit above does not rely on $(\kappa_n)_{n\in\mathbb{N}}$. In fact, suppose that $(\rho_n)_{n\in\mathbb{N}}\subseteq (0,T^*)$ is such that $\rho_n\nearrow T^*$, as $n\rightarrow \infty$. Thus, by following the process above, one obtains
\begin{align}\label{wt7}
\lim_{n\rightarrow \infty} \|(u,\theta)(\rho_n)- (u^{(2)},\theta^{(2)})\|_{\mathcal{X}_{a,\sigma}^{s}}=0,
\end{align}
for some $(u^{(2)},\theta^{(2)})\in\mathcal{X}_{a,\sigma}^{s}(\mathbb{R}^3)$. Then, by defining $(\varsigma_n)_{n\in\mathbb{N}}\subseteq (0,T^*)$ by $\varsigma_{2n}=\kappa_n$ and $\varsigma_{2n-1}= \rho_n$, for all $n\in \mathbb{N}$, one verifies that $\varsigma_n\nearrow T^*$, as $n\rightarrow\infty$. Thereby, by reanalyzing the arguments established above, we conclude that there exists $(u^{(3)}, \theta^{(3)}) \in \mathcal{X}_{a,\sigma}^{s}(\mathbb{R}^3)$ such that
\begin{align}\label{wt8}
\lim_{n\rightarrow \infty} \|(u,\theta)(\varsigma_n)- (u^{(3)},\theta^{(3)})\|_{\mathcal{X}_{a,\sigma}^{s}}=0.
\end{align}
As a result, the limits (\ref{wt6}), (\ref{wt7}) and (\ref{wt8}) imply that $(u^{(1)}, \theta^{(1)}) = (u^{(3)}, \theta^{(3)}) = (u^{(2)}, \theta^{(2)})$. Therefore, it follows that
\begin{align}\label{n5}
\lim_{t\nearrow T^*} \|(u,\theta)(t)- (u^{(1)},\theta^{(1)})\|_{\mathcal{X}_{a,\sigma}^{s}}=0.
\end{align}
Hence, by taking $(u^{(1)},\theta^{(1)})$ as initial data for the Boussinesq equations (\ref{MHD-alpha}), we obtain (see Theorems \ref{teoremaexistenciaB2} and \ref{teoremaexistenciaB}) the local existence and uniqueness of $(\bar{u},\bar{\theta})\in C_{\bar{T}}(\mathcal{X}_{a,\sigma}^{s}(\mathbb{R}^3))$ ($\bar{T}>0$) for this system. Thus, $(\tilde{u},\tilde{\theta})\in C_{\bar{T}+T^*}(\mathcal{X}_{a,\sigma}^{s}(\mathbb{R}^3))$ (see (\ref{n5})) given by
$$(\tilde{u},\tilde{\theta})(t)=\left\{
\begin{array}{ll}
(u,\theta)(t), & t\in[0,T^*); \\
(\bar{u},\bar{\theta})(t-T^*), & t\in[T^*,\bar{T}+T^*],
\end{array}
\right.
$$
solves the Boussinesq equations (\ref{MHD-alpha}), with initial data $(u_0,\theta_0)$, in $[0,T^*+\bar{T}]$. This is a contradiction.
\caixa

\noindent\textbf{Proof of Theorem \ref{teoremaB2} ii):} Analogously to (\ref{ls4}) and (\ref{ls4b}), we obtain
\begin{align*}
\nonumber&\|u(T)\|_{\mathcal{X}_{a,\sigma}^{s}} + \int_{t}^{T}\|u(\tau)\|_{\mathcal{X}_{a,\sigma}^{s+2\alpha}}\;d\tau\leq \|u(t)\|_{\mathcal{X}_{a,\sigma}^{s}}
+C_{s}\int_{t}^{T}\|u\|_{\mathcal{X}_{a,\sigma}^{s}}^{2+\frac{s-1}{2\alpha}}\|u\|_{\mathcal{X}_{a,\sigma}^{s+2\alpha}}^{\frac{1-s}{2\alpha}} \;d\tau+\int_{t}^{T} \|\theta\|_{\mathcal{X}_{a,\sigma}^{s}}\;d\tau
\end{align*}
and also
\begin{align*}
\nonumber&\|\theta(T)\|_{\mathcal{X}_{a,\sigma}^{s}} + \int_{t}^{T}\|\theta(\tau)\|_{\mathcal{X}_{a,\sigma}^{s+2\beta}}\;d\tau\leq \|\theta(t)\|_{\mathcal{X}_{a,\sigma}^{s}}\\
&+ C_{s}\int_{t}^{T}[\|u\|_{\mathcal{X}_{a,\sigma}^{s}}^{1+\frac{s}{2\alpha}}\|u\|_{\mathcal{X}_{a,\sigma}^{s+2\alpha}}^{-\frac{s}{2\alpha}}\|\theta\|_{\mathcal{X}_{a,\sigma}^{s}}^{1-\frac{1}{2\beta}}
\|\theta\|_{\mathcal{X}_{a,\sigma}^{s+2\beta}}^{\frac{1}{2\beta}} +
\|u\|_{\mathcal{X}_{a,\sigma}^{s}}^{1-\frac{1}{2\alpha}}\|u\|_{\mathcal{X}_{a,\sigma}^{s+2\alpha}}^{\frac{1}{2\alpha}}\|\theta\|_{\mathcal{X}_{a,\sigma}^{s}}^{1+\frac{s}{2\beta}}
\|\theta\|_{\mathcal{X}_{a,\sigma}^{s+2\beta}}^{-\frac{s}{2\beta}}
]d\tau,
\end{align*}
for all $0\leq t\leq T<T^*$, provided that $a\geq0, \sigma\geq 1$, $\max\{-1,-2\alpha,-2\beta\}\leq s\leq 0$, $\alpha \geq\frac{1}{2}$ and $\beta\geq \frac{1}{2}$. By adding the two last inequalities and applying Young's inequality, one concludes
\begin{align*}
&\|(u,\theta)(T)\|_{\mathcal{X}_{a,\sigma}^{s}} + \frac{1}{2}\int_{t}^{T}[\|u(\tau)\|_{\mathcal{X}_{a,\sigma}^{s+2\alpha}}+\|\theta(\tau)\|_{\mathcal{X}_{a,\sigma}^{s+2\beta}}]\;d\tau\leq \|(u,\theta)(t)\|_{\mathcal{X}_{a,\sigma}^{s}}\\
&+
C_{s,\alpha,\beta}\int_{t}^{T}\|(u,\theta)\|_{\mathcal{X}_{a,\sigma}^{s}}[\|u\|_{\mathcal{X}_{a,\sigma}^{s}}^{\frac{\beta(s+2\alpha)}{\beta(s+2\alpha)-\alpha}}
\|\theta\|_{\mathcal{X}_{a,\sigma}^{s}}^{\frac{-s\beta}{\beta(s+2\alpha)-\alpha}}
+\|u\|_{\mathcal{X}_{a,\sigma}^{s}}^{\frac{-s\alpha}{\alpha(s+2\beta)-\beta}}\|\theta\|_{\mathcal{X}_{a,\sigma}^{s}}^{\frac{\alpha(s+2\beta)}{\alpha(s+2\beta)-\beta}}
+
\|u\|_{\mathcal{X}_{a,\sigma}^{s}}^{\frac{2\alpha}{s+2\alpha-1}}+1]d\tau,
\end{align*}
for all $0\leq t\leq T<T^*$, provided that $s>\max\big\{1-2\alpha,\frac{\alpha(1-2\beta)}{\beta},\frac{\beta(1-2\alpha)}{\alpha}\big\} $, $\alpha>\frac{1}{2}$ and $\beta> \frac{1}{2}$. According to Gronwall's inequality, it follows that
\begin{align}\label{c3}
\nonumber&\|(u,\theta)(T)\|_{\mathcal{X}_{a,\sigma}^{s}}\leq e^{C_{s,\alpha,\beta}(T-t)}\|(u,\theta)(t)\|_{\mathcal{X}_{a,\sigma}^{s}}\\
&\times\exp\Big\{C_{s,\alpha,\beta}\int_{t}^{T}[\|u\|_{\mathcal{X}_{a,\sigma}^{s}}^{\frac{\beta(s+2\alpha)}{\beta(s+2\alpha)-\alpha}}
\|\theta\|_{\mathcal{X}_{a,\sigma}^{s}}^{\frac{-s\beta}{\beta(s+2\alpha)-\alpha}}
+\|u\|_{\mathcal{X}_{a,\sigma}^{s}}^{\frac{-s\alpha}{\alpha(s+2\beta)-\beta}}\|\theta\|_{\mathcal{X}_{a,\sigma}^{s}}^{\frac{\alpha(s+2\beta)}{\alpha(s+2\beta)-\beta}}
+
\|u\|_{\mathcal{X}_{a,\sigma}^{s}}^{\frac{2\alpha}{s+2\alpha-1}}]d\tau\Big\},
\end{align}
for all $0\leq t\leq T<T^*$. By passing to the limit, as  $T\nearrow T^*$, and using Theorem \ref{teoremaB2} i), one infers
\begin{align*}
\int_{t}^{T^*}[\|u\|_{\mathcal{X}_{a,\sigma}^{s}}^{\frac{\beta(s+2\alpha)}{\beta(s+2\alpha)-\alpha}}
\|\theta\|_{\mathcal{X}_{a,\sigma}^{s}}^{\frac{-s\beta}{\beta(s+2\alpha)-\alpha}}
+\|u\|_{\mathcal{X}_{a,\sigma}^{s}}^{\frac{-s\alpha}{\alpha(s+2\beta)-\beta}}\|\theta\|_{\mathcal{X}_{a,\sigma}^{s}}^{\frac{\alpha(s+2\beta)}{\alpha(s+2\beta)-\beta}}
+
\|u\|_{\mathcal{X}_{a,\sigma}^{s}}^{\frac{2\alpha}{s+2\alpha-1}}]d\tau=\infty,
\end{align*}
for all $t\in[0,T^*)$. 

\caixa

\noindent\textbf{Proof of Theorem \ref{teoremaB1} ii):} Similarly to (\ref{ls4wp}) and (\ref{ls4bwp}), we can write
\begin{align*}
\nonumber&\|u(T)\|_{\mathcal{X}_{a,\sigma}^{s}} + \int_{t}^{T}\|u(\tau)\|_{\mathcal{X}_{a,\sigma}^{s+2\alpha}}\;d\tau\leq \|u(t)\|_{\mathcal{X}_{a,\sigma}^{s}}
+C_{a,\sigma,s}\int_{t}^{T}\|u\|_{\mathcal{X}_{a,\sigma}^{s}}^{2-\frac{1}{2\alpha}}\|u\|_{\mathcal{X}_{a,\sigma}^{s+2\alpha}}^{\frac{1}{2\alpha}} \;d\tau+\int_{t}^{T} \|\theta\|_{\mathcal{X}_{a,\sigma}^{s}}\;d\tau
\end{align*}
and also
\begin{align*}
\nonumber&\|\theta(T)\|_{\mathcal{X}_{a,\sigma}^{s}} + \int_{t}^{T}\|\theta(\tau)\|_{\mathcal{X}_{a,\sigma}^{s+2\beta}}\;d\tau\leq \|\theta(t)\|_{\mathcal{X}_{a,\sigma}^{s}}\\
&+ C_{a,\sigma,s}\int_{t}^{T}[\|\theta\|_{\mathcal{X}_{a,\sigma}^{s}}\|u\|_{\mathcal{X}_{a,\sigma}^{s}}^{1-\frac{1}{2\alpha}}\|u\|_{\mathcal{X}_{a,\sigma}^{s+2\alpha}}^{\frac{1}{2\alpha}}
+
\|u\|_{\mathcal{X}_{a,\sigma}^{s}}\|\theta\|_{\mathcal{X}_{a,\sigma}^{s}}^{1-\frac{1}{2\beta}}
\|\theta\|_{\mathcal{X}_{a,\sigma}^{s+2\beta}}^{\frac{1}{2\beta}}
]d\tau,
\end{align*}
for all $0\leq t\leq T<T^*$, provided that  $(a,s,\sigma)\in \{(0,\infty)$ $\times[-1,0)\times (1,\infty) \}\cup \{[0,\infty)\times\{0\}\times [1,\infty) \}$, $\alpha\geq\frac{1}{2}$ and $\beta\geq\frac{1}{2}$. By adding the two  inequalities above and applying Young's inequality, one concludes
\begin{align*}
&\|(u,\theta)(T)\|_{\mathcal{X}_{a,\sigma}^{s}} + \frac{1}{2}\int_{t}^{T}[\|u(\tau)\|_{\mathcal{X}_{a,\sigma}^{s+2\alpha}}+\|\theta(\tau)\|_{\mathcal{X}_{a,\sigma}^{s+2\beta}}]\;d\tau\leq \|(u,\theta)(t)\|_{\mathcal{X}_{a,\sigma}^{s}}\\
&+
C_{a,\sigma,s,\alpha,\beta}\int_{t}^{T}\|(u,\theta)\|_{\mathcal{X}_{a,\sigma}^{s}}[\|(u,\theta)\|_{\mathcal{X}_{a,\sigma}^{s}}^{\frac{2\alpha}{2\alpha-1}}
+\|u\|_{\mathcal{X}_{a,\sigma}^{s}}^{\frac{2\beta}{2\beta-1}}
+1]d\tau,
\end{align*}
for all $0\leq t\leq T<T^*$, provided that $\alpha>\frac{1}{2}$ and $\beta> \frac{1}{2}$. By using Gronwall's inequality, one deduces
\begin{align}\label{c3pn}
\|(u,\theta)(T)\|_{\mathcal{X}_{a,\sigma}^{s}}\leq e^{C_{a,\sigma,s,\alpha,\beta}(T-t)}\|(u,\theta)(t)\|_{\mathcal{X}_{a,\sigma}^{s}}
\exp\Big\{C_{a,\sigma,s,\alpha,\beta}\int_{t}^{T}[\|(u,\theta)\|_{\mathcal{X}_{a,\sigma}^{s}}^{\frac{2\alpha}{2\alpha-1}}
+\|u\|_{\mathcal{X}_{a,\sigma}^{s}}^{\frac{2\beta}{2\beta-1}}]d\tau\Big\},
\end{align}
for all $0\leq t\leq T<T^*$. By passing to the limit, as  $T\nearrow T^*$, and using Theorem \ref{teoremaB1} i), one infers
\begin{align*}
\int_{t}^{T^*}[\|(u,\theta)\|_{\mathcal{X}_{a,\sigma}^{s}}^{\frac{2\alpha}{2\alpha-1}}
+\|u\|_{\mathcal{X}_{a,\sigma}^{s}}^{\frac{2\beta}{2\beta-1}}]d\tau=\infty, \quad \forall  t\in[0,T^*).
\end{align*}
\caixa

\noindent\textbf{Proof of Theorem \ref{teoremaB2}  iii):} From (\ref{c3}), we can write the following inequality:
\begin{align*}
\nonumber&\|u(T)\|_{\mathcal{X}_{a,\sigma}^{s}}^{\frac{\beta(s+2\alpha)}{\beta(s+2\alpha)-\alpha}}
\|\theta(T)\|_{\mathcal{X}_{a,\sigma}^{s}}^{\frac{-s\beta}{\beta(s+2\alpha)-\alpha}}
+\|u(T)\|_{\mathcal{X}_{a,\sigma}^{s}}^{\frac{-s\alpha}{\alpha(s+2\beta)-\beta}}\|\theta(T)\|_{\mathcal{X}_{a,\sigma}^{s}}^{\frac{\alpha(s+2\beta)}{\alpha(s+2\beta)-\beta}}
+
\|u(T)\|_{\mathcal{X}_{a,\sigma}^{s}}^{\frac{2\alpha}{s+2\alpha-1}} \\
&\leq e^{C_{s,\alpha,\beta}(T-t)}[\|(u,\theta)(t)\|_{\mathcal{X}_{a,\sigma}^{s}}^{\frac{2\alpha}{2\alpha+s-1}}+\|(u,\theta)(t)\|_{\mathcal{X}_{a,\sigma}^{s}}^{\frac{2\alpha\beta}{2\alpha\beta+s\beta-\alpha}}
+\|(u,\theta)(t)\|_{\mathcal{X}_{a,\sigma}^{s}}^{\frac{2\alpha\beta}{2\alpha\beta+s\alpha-\beta}}]\\
&\times \exp\Big\{C_{s,\alpha,\beta}\int_{t}^{T}[\|u\|_{\mathcal{X}_{a,\sigma}^{s}}^{\frac{\beta(s+2\alpha)}{\beta(s+2\alpha)-\alpha}}
\|\theta\|_{\mathcal{X}_{a,\sigma}^{s}}^{\frac{-s\beta}{\beta(s+2\alpha)-\alpha}}
+\|u\|_{\mathcal{X}_{a,\sigma}^{s}}^{\frac{-s\alpha}{\alpha(s+2\beta)-\beta}}\|\theta\|_{\mathcal{X}_{a,\sigma}^{s}}^{\frac{\alpha(s+2\beta)}{\alpha(s+2\beta)-\beta}}
+
\|u\|_{\mathcal{X}_{a,\sigma}^{s}}^{\frac{2\alpha}{s+2\alpha-1}}]d\tau\Big\},
\end{align*}
for all $0\leq t\leq T<T^*$, provided that $a\geq0, \sigma\geq 1$, $ s> \max\{1-2\alpha,\tfrac{\alpha(1-2\beta)}{\beta},\tfrac{\beta(1-2\alpha)}{\alpha}\}$, $ s\in [-1,0]$, $\alpha >\frac{1}{2}$ and $\beta> \frac{1}{2}$. This last result means that
\begin{align*}
&\frac{d}{dT}\Big[-C_{s,\alpha,\beta}^{-1}\exp\Big\{-C_{s,\alpha,\beta}\int_t^T[\|u\|_{\mathcal{X}_{a,\sigma}^{s}}^{\frac{\beta(s+2\alpha)}{\beta(s+2\alpha)-\alpha}}
\|\theta\|_{\mathcal{X}_{a,\sigma}^{s}}^{\frac{-s\beta}{\beta(s+2\alpha)-\alpha}}
+\|u\|_{\mathcal{X}_{a,\sigma}^{s}}^{\frac{-s\alpha}{\alpha(s+2\beta)-\beta}}\|\theta\|_{\mathcal{X}_{a,\sigma}^{s}}^{\frac{\alpha(s+2\beta)}{\alpha(s+2\beta)-\beta}}
+
\|u\|_{\mathcal{X}_{a,\sigma}^{s}}^{\frac{2\alpha}{s+2\alpha-1}}]d\tau\Big\}\Big]\\
&\leq e^{C_{s,\alpha,\beta}(T-t)}[\|(u,\theta)(t)\|_{\mathcal{X}_{a,\sigma}^{s}}^{\frac{2\alpha}{2\alpha+s-1}}+\|(u,\theta)(t)\|_{\mathcal{X}_{a,\sigma}^{s}}^{\frac{2\alpha\beta}{2\alpha\beta+s\beta-\alpha}}
+\|(u,\theta)(t)\|_{\mathcal{X}_{a,\sigma}^{s}}^{\frac{2\alpha\beta}{2\alpha\beta+s\alpha-\beta}}].
\end{align*}
By integrating  over $[t,t_0]$, with $0 \leq t \leq t_0 < T^*$, the inequality above, one has
\begin{align*}
&-\exp\Big\{-C_{s,\alpha,\beta}\int_t^{t_0}[\|u\|_{\mathcal{X}_{a,\sigma}^{s}}^{\frac{\beta(s+2\alpha)}{\beta(s+2\alpha)-\alpha}}
\|\theta\|_{\mathcal{X}_{a,\sigma}^{s}}^{\frac{-s\beta}{\beta(s+2\alpha)-\alpha}}
+\|u\|_{\mathcal{X}_{a,\sigma}^{s}}^{\frac{-s\alpha}{\alpha(s+2\beta)-\beta}}\|\theta\|_{\mathcal{X}_{a,\sigma}^{s}}^{\frac{\alpha(s+2\beta)}{\alpha(s+2\beta)-\beta}}
+
\|u\|_{\mathcal{X}_{a,\sigma}^{s}}^{\frac{2\alpha}{s+2\alpha-1}}]d\tau\Big\}+1\\
&\leq 
[e^{C_{s,\alpha,\beta}(t_0-t)}-1][\|(u,\theta)(t)\|_{\mathcal{X}_{a,\sigma}^{s}}^{\frac{2\alpha}{2\alpha+s-1}}+\|(u,\theta)(t)\|_{\mathcal{X}_{a,\sigma}^{s}}^{\frac{2\alpha\beta}{2\alpha\beta+s\beta-\alpha}}
+\|(u,\theta)(t)\|_{\mathcal{X}_{a,\sigma}^{s}}^{\frac{2\alpha\beta}{2\alpha\beta+s\alpha-\beta}}].
\end{align*}
for all $0 \leq t \leq t_0 < T^*$. Therefore, by passing to the limit, as $t_0\nearrow T^*$, Theorem \ref{teoremaB2} ii) leads us to
\begin{align*}
\|(u,\theta)(t)\|_{\mathcal{X}_{a,\sigma}^{s}}^{\frac{2\alpha}{2\alpha+s-1}}+\|(u,\theta)(t)\|_{\mathcal{X}_{a,\sigma}^{s}}^{\frac{2\alpha\beta}{2\alpha\beta+s\beta-\alpha}}
+\|(u,\theta)(t)\|_{\mathcal{X}_{a,\sigma}^{s}}^{\frac{2\alpha\beta}{2\alpha\beta+s\alpha-\beta}}\geq [e^{C_{s,\alpha,\beta}(T^*-t)}-1]^{-1},
\end{align*}
for all $t\in [0,T^*)$.
\caixa

\noindent\textbf{Proof of Theorem  \ref{teoremaB1} iii):} 
By observing (\ref{c3pn}), we have
\begin{align*}
\|(u,\theta)(T)\|_{\mathcal{X}_{a,\sigma}^{s}}^{\frac{2\alpha}{2\alpha-1}}+\|u(T)\|_{\mathcal{X}_{a,\sigma}^{s}}^{\frac{2\beta}{2\beta-1}}&\leq e^{C_{a,\sigma,s,\alpha,\beta}(T-t)}[\|(u,\theta)(t)\|_{\mathcal{X}_{a,\sigma}^{s}}^{\frac{2\alpha}{2\alpha-1}}+\|(u,\theta)(t)\|_{\mathcal{X}_{a,\sigma}^{s}}^{\frac{2\beta}{2\beta-1}}]\\
&\quad\times \exp\Big\{C_{a,\sigma,s,\alpha,\beta}\int_{t}^{T}[\|(u,\theta)\|_{\mathcal{X}_{a,\sigma}^{s}}^{\frac{2\alpha}{2\alpha-1}}
+\|u\|_{\mathcal{X}_{a,\sigma}^{s}}^{\frac{2\beta}{2\beta-1}}]d\tau\Big\},
\end{align*}
for all $0\leq t\leq T<T^*$, provided that $(a,s,\sigma)\in \{(0,\infty)$ $\times[-1,0)\times (1,\infty) \}\cup \{[0,\infty)\times\{0\}\times [1,\infty) \}$, $\alpha>\frac{1}{2}$ and $\beta>\frac{1}{2}$. Consequently, one obtains
\begin{align*}
&\frac{d}{dT}\Big[-C_{a,\sigma,s,\alpha,\beta}^{-1}\exp\Big\{-C_{a,\sigma,s,\alpha,\beta}\int_t^T[\|(u,\theta)\|_{\mathcal{X}_{a,\sigma}^{s}}^{\frac{2\alpha}{2\alpha-1}}
+\|u\|_{\mathcal{X}_{a,\sigma}^{s}}^{\frac{2\beta}{2\beta-1}}]d\tau\Big\}\Big]\\
&\leq e^{C_{a,\sigma,s,\alpha,\beta}(T-t)}[\|(u,\theta)(t)\|_{\mathcal{X}_{a,\sigma}^{s}}^{\frac{2\alpha}{2\alpha-1}}
+\|(u,\theta)(t)\|_{\mathcal{X}_{a,\sigma}^{s}}^{\frac{2\beta}{2\beta-1}}].
\end{align*}
By integrating  over $[t,t_0]$, with $0 \leq t \leq t_0 < T^*$, the inequality above, we deduce
\begin{align*}
&-\exp\Big\{-C_{a,\sigma,s,\alpha,\beta}\int_t^{t_0}[\|(u,\theta)\|_{\mathcal{X}_{a,\sigma}^{s}}^{\frac{2\alpha}{2\alpha-1}}
+\|u\|_{\mathcal{X}_{a,\sigma}^{s}}^{\frac{2\beta}{2\beta-1}}]d\tau\Big\}+1\\
&\leq 
[e^{C_{a,\sigma,s,\alpha,\beta}(t_0-t)}-1][\|(u,\theta)(t)\|_{\mathcal{X}_{a,\sigma}^{s}}^{\frac{2\alpha}{2\alpha-1}}
+\|(u,\theta)(t)\|_{\mathcal{X}_{a,\sigma}^{s}}^{\frac{2\beta}{2\beta-1}}].
\end{align*}
By passing to the limit, as $t_0\nearrow T^*$, Theorem \ref{teoremaB1} ii) allows us to conclude that
\begin{align*}
\|(u,\theta)(t)\|_{\mathcal{X}_{a,\sigma}^{s}}^{\frac{2\alpha}{2\alpha-1}}
+\|(u,\theta)(t)\|_{\mathcal{X}_{a,\sigma}^{s}}^{\frac{2\beta}{2\beta-1}}\geq [e^{C_{a,\sigma,s,\alpha,\beta}(T^*-t)}-1]^{-1}, \quad \forall t\in[0,T^*).
\end{align*}
\caixa

\noindent\textbf{Proof of Corollary \ref{corollaryB1} i)--iv):} By applying the arguments developed in the inequalities (\ref{ls3}) and (\ref{ls3b}), and by using Lemmas \ref{lema2} and \ref{lema1},  we can conclude
\begin{align*}
&\|(u,\theta)(T)\|_{\mathcal{X}_{a,\sigma}^{s}} + \int_{t}^{T}\|u(\tau)\|_{\mathcal{X}_{a,\sigma}^{s+2\alpha}}\;d\tau+ \int_{t}^{T}\|\theta(\tau)\|_{\mathcal{X}_{a,\sigma}^{s+2\beta}}\;d\tau\leq \|(u,\theta)(t)\|_{\mathcal{X}_{a,\sigma}^{s}}\\
&+C_s\int_{t}^{T}[\|u\|_{\mathcal{X}_{\frac{a}{\sigma},\sigma}^{0}}\|u\|_{\mathcal{X}_{a,\sigma}^{s}}^{1-\frac{1}{2\alpha}}\|u\|_{\mathcal{X}_{a,\sigma}^{s+2\alpha}}^{\frac{1}{2\alpha}}
+\|\theta\|_{\mathcal{X}_{\frac{a}{\sigma},\sigma}^{0}}\|u\|_{\mathcal{X}_{a,\sigma}^{s}}^{1-\frac{1}{2\alpha}}\|u\|_{\mathcal{X}_{a,\sigma}^{s+2\alpha}}^{\frac{1}{2\alpha}}+\|u\|_{\mathcal{X}_{\frac{a}{\sigma},\sigma}^{0}}
\|\theta\|_{\mathcal{X}_{a,\sigma}^{s}}^{1-\frac{1}{2\beta}}\|\theta\|_{\mathcal{X}_{a,\sigma}^{s+2\beta}}^{\frac{1}{2\beta}}]\;d\tau\\
&+ \int_{t}^{T}\|\theta(\tau)\|_{\mathcal{X}_{a,\sigma}^{s}}\;d\tau,
\end{align*}
for all $t\in [0,T^*)$, provided that $a\geq0$, $\sigma\geq1$, $s\geq-1$, $\alpha\geq\frac{1}{2}$ and $\beta\geq \frac{1}{2}$. Now, apply Young's inequality  in order to reach the inequality below:
\begin{align*}
&\|(u,\theta)(T)\|_{\mathcal{X}_{a,\sigma}^{s}} + \frac{1}{2}\int_{t}^{T}\|u(\tau)\|_{\mathcal{X}_{a,\sigma}^{s+2\alpha}}\;d\tau+ \frac{1}{2}\int_{t}^{T}\|\theta(\tau)\|_{\mathcal{X}_{a,\sigma}^{s+2\beta}}\;d\tau\leq \|(u,\theta)(t)\|_{\mathcal{X}_{a,\sigma}^{s}}\\
&+C_{s,\alpha,\beta}\int_{t}^{T}\|(u,\theta)\|_{\mathcal{X}_{a,\sigma}^{s}}[\|(u,\theta)\|_{\mathcal{X}_{\frac{a}{\sigma},\sigma}^{0}}^{\frac{2\alpha}{2\alpha-1}}
+\|u\|_{\mathcal{X}_{\frac{a}{\sigma},\sigma}^{0}}^{\frac{2\beta}{2\beta-1}}+1]\;d\tau,
\end{align*}
for all $t\in [0,T^*)$, provided that $\alpha>\frac{1}{2}$ and $\beta> \frac{1}{2}$. By applying  Gronwall's inequality, we can write the following inequality:
\begin{align}\label{n1}
\|(u,\theta)(T)\|_{\mathcal{X}_{a,\sigma}^{s}}&\leq 
 e^{C_{s,\alpha,\beta}(T-t)}\|(u,\theta)(t)\|_{\mathcal{X}_{a,\sigma}^{s}}\exp \{C_{s,\alpha,\beta}\int_t^T [\|(u,\theta)\|_{\mathcal{X}_{\frac{a}{\sigma},\sigma}^{0}}^{\frac{2\alpha}{2\alpha-1}}
+\|u\|_{\mathcal{X}_{\frac{a}{\sigma},\sigma}^{0}}^{\frac{2\beta}{2\beta-1}}]\,d\tau \},
\end{align}
for all $0\leq t\leq T< T^*$. By passing to the limit superior, as $T\nearrow T^*$, Theorem \ref{teoremaB1}  i) implies that
\begin{align}\label{n2}
\int_t^{T^*} [\|(u,\theta)\|_{\mathcal{X}_{\frac{a}{\sigma},\sigma}^{0}}^{\frac{2\alpha}{2\alpha-1}}
+\|u\|_{\mathcal{X}_{\frac{a}{\sigma},\sigma}^{0}}^{\frac{2\beta}{2\beta-1}}]\,d\tau=\infty, \quad \forall t\in[0,T^*),
\end{align}
provided that $\alpha>\frac{1}{2}$, $\beta>\frac{1}{2}$ and $(a,s,\sigma)\in\{(0,\infty)\times [-1,0)\times (1,\infty)\}\cup \{[0,\infty)\times\{0\}\times [1,\infty)\}$. Therefore,  Corollary  \ref{corollaryB1} i), with $n=1$, is proved.

Similarly to (\ref{n1}), it follows the inequality below: 
\begin{align*}
\|(u,\theta)(T)\|_{\mathcal{X}_{\frac{a}{\sigma},\sigma}^{0}}&\leq 
 e^{C_{\alpha,\beta}(T-t)}\|(u,\theta)(t)\|_{\mathcal{X}_{\frac{a}{\sigma},\sigma}^{0}}\exp \{C_{\alpha,\beta}\int_t^T [\|(u,\theta)\|_{\mathcal{X}_{\frac{a}{\sigma},\sigma}^{0}}^{\frac{2\alpha}{2\alpha-1}}
+\|u\|_{\mathcal{X}_{\frac{a}{\sigma},\sigma}^{0}}^{\frac{2\beta}{2\beta-1}}]\,d\tau \},
\end{align*}
for all $t\in [0,T^*)$, provided that $a\geq0$, $\sigma\geq1$, $\alpha>\frac{1}{2}$ and $\beta>\frac{1}{2}$.  Thereby, one obtains
\begin{align*}
\|(u,\theta)(T)\|_{\mathcal{X}_{\frac{a}{\sigma},\sigma}^{0}}^{\frac{2\alpha}{2\alpha-1}}+\|u(T)\|_{\mathcal{X}_{\frac{a}{\sigma},\sigma}^{0}}^{\frac{2\beta}{2\beta-1}}&\leq e^{C_{\alpha,\beta}(T-t)}[\|(u,\theta)(t)\|_{\mathcal{X}_{\frac{a}{\sigma},\sigma}^{0}}^{\frac{2\alpha}{2\alpha-1}}+\|(u,\theta)(t)\|_{\mathcal{X}_{\frac{a}{\sigma},\sigma}^{0}}^{\frac{2\beta}{2\beta-1}}]\\
&\quad\times \exp\Big\{C_{\alpha,\beta}\int_{t}^{T}[\|(u,\theta)\|_{\mathcal{X}_{\frac{a}{\sigma},\sigma}^{0}}^{\frac{2\alpha}{2\alpha-1}}
+\|u\|_{\mathcal{X}_{\frac{a}{\sigma},\sigma}^{0}}^{\frac{2\beta}{2\beta-1}}]d\tau\Big\},
\end{align*}
for all $0\leq t\leq T<T^*$. As a result, we infer
\begin{align*}
&\frac{d}{dT}\Big[-C_{\alpha,\beta}^{-1}\exp\Big\{-C_{\alpha,\beta}\int_t^T[\|(u,\theta)\|_{\mathcal{X}_{\frac{a}{\sigma},\sigma}^{0}}^{\frac{2\alpha}{2\alpha-1}}
+\|u\|_{\mathcal{X}_{\frac{a}{\sigma},\sigma}^{0}}^{\frac{2\beta}{2\beta-1}}]d\tau\Big\}\Big]\leq e^{C_{\alpha,\beta}(T-t)}[\|(u,\theta)(t)\|_{\mathcal{X}_{\frac{a}{\sigma},\sigma}^{0}}^{\frac{2\alpha}{2\alpha-1}}
+\|(u,\theta)(t)\|_{\mathcal{X}_{\frac{a}{\sigma},\sigma}^{0}}^{\frac{2\beta}{2\beta-1}}],
\end{align*}
for all $0\leq t\leq T<T^*$. Integrate  over $[t,t_0]$, with $0 \leq t \leq t_0 < T^*$, to obtain
\begin{align*}
&-\exp\Big\{-C_{\alpha,\beta}\int_t^{t_0}[\|(u,\theta)\|_{\mathcal{X}_{\frac{a}{\sigma},\sigma}^{0}}^{\frac{2\alpha}{2\alpha-1}}
+\|u\|_{\mathcal{X}_{\frac{a}{\sigma},\sigma}^{0}}^{\frac{2\beta}{2\beta-1}}]d\tau\Big\}+1\leq 
[e^{C_{\alpha,\beta}(t_0-t)}-1][\|(u,\theta)(t)\|_{\mathcal{X}_{\frac{a}{\sigma},\sigma}^{0}}^{\frac{2\alpha}{2\alpha-1}}
+\|(u,\theta)(t)\|_{\mathcal{X}_{\frac{a}{\sigma},\sigma}^{0}}^{\frac{2\beta}{2\beta-1}}].
\end{align*}
By taking the limit, as $t_0\nearrow T^*$, Corollary \ref{corollaryB1} i), with $n=1$, implies that
\begin{align}\label{n3}
\|(u,\theta)(t)\|_{\mathcal{X}_{\frac{a}{\sigma},\sigma}^{0}}^{\frac{2\alpha}{2\alpha-1}}
+\|(u,\theta)(t)\|_{\mathcal{X}_{\frac{a}{\sigma},\sigma}^{0}}^{\frac{2\beta}{2\beta-1}}\geq [e^{C_{\alpha,\beta}(T^*-t)}-1]^{-1}, \quad \forall t\in[0,T^*),
\end{align}
 provided that $\alpha>\frac{1}{2}$, $\beta>\frac{1}{2}$, $a\geq0$ and $\sigma\geq1$. Therefore,  Corollary  \ref{corollaryB1} ii), with $n=1$, is proved.

On the other hand, note that  $\mathcal{X}_{a,\sigma}^s(\mathbb{R}^3)\hookrightarrow \mathcal{X}_{\frac{a}{(\sqrt{\sigma})^n},\sigma}^s(\mathbb{R}^3)$, for all $n\in\mathbb{N}$
(since $\sigma\geq  1$). More accurately, the inequality $\displaystyle\|\cdot\|_{\mathcal{X}_{\frac{a}{(\sqrt{\sigma})^n},\sigma}^s}\leq \|\cdot\|_{\mathcal{X}_{a,\sigma}^s}$ holds. As a result, we have$\footnote{From now on $T^*_{\omega}< \infty$ denotes the first blow-up time for the solution $(u,\theta)\in C([0,T^*_{\omega}); \mathcal{X}_{\omega,\sigma}^{s}(\mathbb{R}^3))$, where $\omega>0.$}$ $(u,\theta)\in C([0,T_{a}^*),$ $\mathcal{X}_{\frac{a}{(\sqrt{\sigma})^n},\sigma}^s(\mathbb{R}^3))$ (recall that $(u,\theta)\in C([0,T_{a}^*),$ $\mathcal{X}_{a,\sigma}^s(\mathbb{R}^3))$) and, consequently, 
\begin{align}\label{wilber10}
T_{\frac{a}{(\sqrt{\sigma})^n}}^*\geq T_a^*,\quad\forall n\in\mathbb{N}.
\end{align}
Moreover, from Corollary \ref{corollaryB1} ii) (with $n=1$) and Lemma \ref{lema0} (with $\mu=\sqrt{\sigma}$), one deduces
\begin{align*}
\nonumber [e^{C_{\alpha,\beta}(T^{*}_a-t)}-1]^{-1}&\leq  \| (u,\theta)(t)\|_{\mathcal{X}^0_{\frac{a}{\sigma},\sigma}}^{\frac{2\alpha}{2\alpha-1}}+\| (u,\theta)(t)\|_{\mathcal{X}^0_{\frac{a}{\sigma},\sigma}}^{\frac{2\beta}{2\beta-1}}\\
&\leq C_{a,\sigma,s,\alpha,\beta}[\|(u,\theta)(t)\|_{\mathcal{X}_{\frac{a}{\sqrt{\sigma}},\sigma}^s}^{\frac{2\alpha}{2\alpha-1}}+\|(u,\theta)(t)\|_{\mathcal{X}_{\frac{a}{\sqrt{\sigma}},\sigma}^s}^{\frac{2\beta}{2\beta-1}}],
\end{align*}
for all $t\in [0,T_a^*)$, provided that  $a>0$, $\sigma>1$ and $s\leq0$. This means that
\begin{align}\label{i)n=1}
 \|(u,\theta)(t)\|_{\mathcal{X}_{\frac{a}{\sqrt{\sigma}},\sigma}^s}^{\frac{2\alpha}{2\alpha-1}}+\|(u,\theta)(t)\|_{\mathcal{X}_{\frac{a}{\sqrt{\sigma}},\sigma}^s}^{\frac{2\beta}{2\beta-1}}\geq C_{a,\sigma,s,\alpha,\beta}[e^{C_{\alpha,\beta}(T^{*}_a-t)}-1]^{-1},
\end{align}
for all $t\in [0,T_a^*)$. Therefore, the proof of Corollary \ref{corollaryB1} iii), with $n=1$, is established.

It is easy to see that  (\ref{i)n=1})  implies the next limit: 
\begin{align}\label{esqueci2}
\limsup_{t\nearrow T_a^*}\|(u,\theta)(t)\|_{\mathcal{X}_{\frac{a}{\sqrt{\sigma}},\sigma}^s}=\infty.
\end{align}
This result establishes the proof of Corollary \ref{corollaryB1} iv), with $n=2$ (the case $n=1$ was proved in Theorems \ref{teoremaB2} i) and \ref{teoremaB1} i) and, in addition, we also have
\begin{align}\label{wilber11a}
T_a^*\geq T_{\frac{a}{\sqrt{\sigma}}}^*.
\end{align}
 By (\ref{wilber10}) and (\ref{wilber11a}), we reach
\begin{align}\label{wilber11}
T_a^*= T_{\frac{a}{\sqrt{\sigma}}}^*.
\end{align}
On the other hand, by replacing $a$ by $\frac{a}{\sqrt{\sigma}}$ in the proof of Corollary \ref{corollaryB1} i), with $n=1$, and applying (\ref{esqueci2}), one obtains the verification of Corollary \ref{corollaryB1} i), with $n=2$, similarly to (\ref{n2}).

 As a consequence, by using this Corollary \ref{corollaryB1} i), with $n=2$,  we can prove Corollary \ref{corollaryB1} ii), with $n=2$, analogously to (\ref{n3}).

By replacing $a$ by $\frac{a}{\sqrt{\sigma}}$ in the proof of (\ref{i)n=1}), applying Corollary \ref{corollaryB1} ii) (with $n=2$) and Lemma \ref{lema0} (with $\mu=\sqrt{\sigma}$), we conclude that
\begin{align*}
 \|(u,\theta)(t)\|_{\mathcal{X}_{\frac{a}{\sigma},\sigma}^s}^{\frac{2\alpha}{2\alpha-1}}
+\|(u,\theta)(t)\|_{\mathcal{X}_{\frac{a}{\sigma},\sigma}^s}^{\frac{2\beta}{2\beta-1}}\geq C_{a,\sigma,s,\alpha,\beta}[e^{C_{\alpha,\beta}(T^{*}_{\frac{a}{\sqrt{\sigma}}}-t)}-1]^{-1},
\end{align*}
for all $t\in [0,T_{\frac{a}{\sqrt{\sigma}}}^*)$. Thus, (\ref{wilber11}) leads us to
\begin{align}\label{n4}
 \|(u,\theta)(t)\|_{\mathcal{X}_{\frac{a}{\sigma},\sigma}^s}^{\frac{2\alpha}{2\alpha-1}}
+\|(u,\theta)(t)\|_{\mathcal{X}_{\frac{a}{\sigma},\sigma}^s}^{\frac{2\beta}{2\beta-1}}\geq C_{a,\sigma,s,\alpha,\beta}[e^{C_{\alpha,\beta}(T^{*}_{a}-t)}-1]^{-1},
\end{align}
for all $t\in [0,T_{a}^*)$. Therefore, we have proved Corollary  \ref{corollaryB1} iii), with $n=2$.

Again, by taking  the limit in  (\ref{n4}), as   $t\nearrow T^*_a$, we can write
$$\displaystyle \limsup_{t\nearrow T_a^*}\|(u,\theta)(t)\|_{\mathcal{X}_{\frac{a}{\sigma},\sigma}^s(\mathbb{R}^3)}=\infty.$$
This limit shows that the proof of Corollary \ref{corollaryB1} iv), with $n=3$, is complete and, moreover, $T^*_a \geq T^*_{\frac{a}{\sigma}}$. By (\ref{wilber10}), one has
$T^*_a = T^*_{\frac{a}{\sigma}}$.

Therefore, inductively, we can prove that $T^*_a = T^*_{\frac{a}{(\sqrt{\sigma})^n}}$, for all $n\in \mathbb{N}$, and Corollary \ref{corollaryB1} i)--iv) hold.
\caixa

\noindent\textbf{Proof of Corollary \ref{corollaryB1} v):} By Corollary \ref{corollaryB1} ii), one has
\begin{align*}
[e^{C_{\alpha,\beta}(T^*-t)}-1]^{-1}&\leq \left(\int_{\mathbb{R}^3}e^{\frac{a}{\sigma(\sqrt{\sigma})^{n-1}}|\xi|^{\frac{1}{\sigma}}}|(\hat{u},\hat{\theta})(t)|\;d\xi\right)^{\frac{2\alpha}{2\alpha-1}}
+\left(\int_{\mathbb{R}^3}e^{\frac{a}{\sigma(\sqrt{\sigma})^{n-1}}|\xi|^{\frac{1}{\sigma}}}|(\hat{u},\hat{\theta})(t)|\;d\xi\right)^{\frac{2\beta}{2\beta-1}},
\end{align*}
for all $t\in[0,T^*)$ and $n\in\mathbb{N}$, provided that $a>0$, $\sigma> 1$, $s\in[-1,0],$ $\alpha>\frac{1}{2}$ and $\beta>\frac{1}{2}$. On the other hand, by Lemma \ref{lema0}, we obtain
$$\|(u,\theta)(t)\|_{\mathcal{X}_{\frac{a}{\sigma},\sigma}^{0}}\leq C_{a,s,\sigma} \|(u,\theta)(t)\|_{\mathcal{X}_{a,\sigma}^{s}}<\infty,\quad\forall t\in[0,T^*),$$
provided that $a>0$, $\sigma>1$ and $s\leq0$. Thus, by passing to the limit, as $n\rightarrow \infty$,  one infers
\begin{align*}
[e^{C_{\alpha,\beta}(T^*-t)}-1]^{-1}&\leq\|(u,\theta)(t)\|_{\mathcal{X}^0}^{\frac{2\alpha}{2\alpha-1}}+ \|(u,\theta)(t)\|_{\mathcal{X}^0}^{\frac{2\beta}{2\beta-1}},\quad\forall t\in[0,T^*).
\end{align*}
\caixa

\noindent\textbf{Proof of Corollary \ref{corollaryB1} vi):} By applying  Corollary \ref{corollaryB1} iii), we obtain
\begin{align*}
C_{a,\sigma,s,\alpha,\beta}[e^{C_{\alpha,\beta}(T^*-t)}-1]^{-1}&\leq \left(\int_{\mathbb{R}^3}|\xi|^se^{\frac{a}{(\sqrt{\sigma})^{n}}|\xi|^{\frac{1}{\sigma}}}|(\hat{u},\hat{\theta})(t)|\;d\xi\right)^{\frac{2\alpha}{2\alpha-1}}
+\left(\int_{\mathbb{R}^3}|\xi|^se^{\frac{a}{(\sqrt{\sigma})^{n}}|\xi|^{\frac{1}{\sigma}}}|(\hat{u},\hat{\theta})(t)|\;d\xi\right)^{\frac{2\beta}{2\beta-1}},
\end{align*}
for all $t\in[0,T^*)$ and $n\in\mathbb{N}$, provided that $a>0$, $\sigma> 1$, $s\in[-1,0],$ $\alpha>\frac{1}{2}$ and $\beta>\frac{1}{2}$. On the other hand, it is easy to check that
$$|\xi|^se^{\frac{a}{(\sqrt{\sigma})^{n}}|\xi|^{\frac{1}{\sigma}}}|(\hat{u},\hat{\theta})(\xi,t)|\leq|\xi|^se^{a|\xi|^{\frac{1}{\sigma}}}|(\hat{u},\hat{\theta})(\xi,t)|\in L^1(\mathbb{R}^3),\quad\forall \xi\in \mathbb{R}^3,t\in[0,T^*),n\in\mathbb{N},$$
provided that $a\geq0$, $\sigma\geq1$ and $s\in \mathbb{R}$.
Hence, by taking the limit, as $n\rightarrow \infty$,  we infer
\begin{align*}
C_{a,\sigma,s,\alpha,\beta}[e^{C_{\alpha,\beta}(T^*-t)}-1]^{-1}&\leq\|(u,\theta)(t)\|_{\mathcal{X}^s}^{\frac{2\alpha}{2\alpha-1}}+ \|(u,\theta)(t)\|_{\mathcal{X}^s}^{\frac{2\beta}{2\beta-1}},\quad\forall t\in[0,T^*).
\end{align*}
\caixa

\noindent\textbf{Proof of Corollary \ref{corollaryB2}:} Observe that, by Corollary \ref{corollaryB1} v), with $\beta=\alpha$, one infers
\begin{align}\label{estimativaalphabeta}
C_{\alpha}[e^{C_{\alpha}(T^*-t)}-1]^{-\frac{2\alpha-1}{2\alpha}}\leq \|(u,\theta)(t)\|_{\mathcal{X}^0}, \quad \forall t\in[0,T^*),
\end{align}
provided that $a>0$, $\sigma>1$, $\alpha>\frac{1}{2}$ and  $s\in[-1,0]$. In order to apply Lemma \ref{lema} to (\ref{estimativaalphabeta}), choose $\delta=s+\frac{k}{\sigma}$, with $\sigma_0$ the integer part of $-s\sigma$, $n_0>\sigma_0$ (natural) and $k\geq n_0$ (natural),   to obtain
\begin{align*}
C_{\alpha}[e^{C_{\alpha}(T^*-t)}-1]^{-\frac{2\alpha-1}{2\alpha}}\leq
C_0\|(u,\theta)(t)\|_{L^2}^{\frac{2(s+\frac{k}{\sigma})}{2(s+\frac{k}{\sigma})+3}}\|(u,\theta)(t)\|_{\mathcal{X}^{s+\frac{k}{\sigma}}}^{\frac{3}{2(s+\frac{k}{\sigma})+3}}, \quad \forall t\in[0,T^*).
\end{align*}
(Notice that the last norm on the right-hand side of the inequality above is finite by Lemma \ref{lemanovo2}). On the other hand, by applying  $L^2$-inner product to the Boussinesq equations (\ref{MHD-alpha}), with $u$ and $\theta$,  and integrating the results obtained over $[0,t]$, it follows that
\begin{align*}
\frac{1}{2}\frac{d}{dt}\|(u,\theta)(t)\|_{L^2}^2+\|(-\Delta)^{\frac{\alpha}{2}}u(t)\|_{L^2}^2+\|(-\Delta)^{\frac{\alpha}{2}}\theta(t)\|_{L^2}^2 &\leq \|(u,\theta)(t)\|_{L^2}^{2},\quad\forall t\in [0,T^*).
\end{align*}
Consequently, by Gronwall's Lemma, one concludes
\begin{align}\label{normal2}
 \|(u,\theta)(t)\|_{L^2}\leq e^{T^*}\|(u_0,\theta_0)\|_{L^2},\quad\forall t\in [0,T^*).
\end{align}
Thus, by using (\ref{normal2}), one deduces that
\begin{align}\label{W1}
C_{T^*}[C_{\alpha,T^*}[e^{C_{\alpha}(T^*-t)}-1]^{-\frac{2\alpha-1}{2\alpha}}\|(u_0,\theta_0)\|_{L^2(\mathbb{R}^3)}^{-1}]^{\frac{2}{3}(s+\frac{k}{\sigma})+1}\|(u_0,\theta_0)\|_{L^2}\leq \|(u,\theta)(t)\|_{\mathcal{X}^{s+\frac{k}{\sigma}}},
\end{align}
for all $t\in[0,T^*)$. Thereby, by multiplying (\ref{W1}) by $\frac{a^k}{k!}$ and summing over $k\geq n_0$ the result obtained, we reach the following inequality:
\begin{align*}
&C_{T^*}[C_{\alpha,T^*}[e^{C_{\alpha}(T^*-t)}-1]^{-\frac{2\alpha-1}{2\alpha}}\|(u_0,\theta_0)\|_{L^2(\mathbb{R}^3)}^{-1}]^{\frac{2s}{3}+1}\|(u_0,\theta_0)\|_{L^2}\\
&\times\sum_{k\geq n_0} \frac{\{a[C_{\alpha,T^*}[e^{C_{\alpha}(T^*-t)}-1]^{-\frac{2\alpha-1}{2\alpha}}\|(u_0,\theta_0)\|_{L^2}^{-1}]^{\frac{2}{3\sigma}}\}^k}{k!}\leq \sum_{k\in\mathbb{N}}\frac{a^k}{k!}\|(u,\theta)(t)\|_{\mathcal{X}^{s+\frac{k}{\sigma}}},
\end{align*}
for all $t\in[0,T^*)$. This means that
\begin{align*}
&C_{T^*}[C_{\alpha,T^*}[e^{C_{\alpha}(T^*-t)}-1]^{-\frac{2\alpha-1}{2\alpha}}\|(u_0,\theta_0)\|_{L^2(\mathbb{R}^3)}^{-1}]^{\frac{2s}{3}+1}\|(u_0,\theta_0)\|_{L^2}\\
&\times\sum_{k\geq n_0} \frac{\{a[C_{\alpha,T^*}[e^{C_{\alpha}(T^*-t)}-1]^{-\frac{2\alpha-1}{2\alpha}}\|(u_0,\theta_0)\|_{L^2}^{-1}]^{\frac{2}{3\sigma}}\}^k}{k!}\leq \|(u,\theta)(t)\|_{\mathcal{X}_{a,\sigma}^{s}},
\end{align*}
for all $t\in[0,T^*)$. On the other hand, define
$$f(x)=\Big[e^x-\displaystyle\sum_{k=0}^{n_0-1}\frac{x^k}{k!}\Big][x^{-n_0}e^{-\frac{x}{2}}],\quad \forall x>0.$$
It is easy to check that  there exists a positive constant $C_{s,\sigma}$ such that $f(x)\geq C_{s,\sigma}$, for all $x>0$. Thus, in particular, by taking $x=a[C_{\alpha,T^*}[e^{C_{\alpha}(T^*-t)}-1]^{-\frac{2\alpha-1}{2\alpha}}\|(u_0,\theta_0)\|_{L^2}^{-1}]^{\frac{2}{3\sigma}}$, we deduce
\begin{align*}
\|(u,\theta)(t)\|_{\mathcal{X}_{a,\sigma}^{s}}&\geq a^{n_0}C_{s,\sigma,\alpha,T^*}[e^{C_{\alpha}(T^*-t)}-1]^{-\frac{2\alpha-1}{2\alpha}(\frac{2s}{3}+\frac{2n_0}{\sigma}+1)}\|(u_0,\theta_0)\|_{L^2}^{-\frac{2}{3}(s+\frac{n_0}{\sigma})}\\
&\quad\times \exp\Big\{aC_{\sigma,\alpha,T^*}
[e^{C_{\alpha}(T^*-t)}-1]^{-\frac{2\alpha-1}{3\alpha\sigma}}\|(u_0,\theta_0)\|_{L^2}^{-\frac{2}{3\sigma}}\Big\},\quad\forall t\in[0,T^*).
\end{align*}
\caixa

\noindent \textbf{Declarations:}\\

\noindent \textbf{Ethics approval and consent to participate:}
Not applicable.\\

\noindent \textbf{Consent for publication:}
Not applicable.\\

\noindent \textbf{Availability of data and materials:}
Not applicable (this manuscript does not report data generation or analysis).\\

\noindent \textbf{Conflicts of interest/Competing interests:}
The authors have no conflicts of interest to declare that are relevant to the content of this article.\\

\noindent \textbf{Funding:}
This work was supported by FAPESP grant 2024/15587-1.\\

\noindent \textbf{Authors' contributions:}
P.G., W.M. and T.S.  wrote and reviewed the manuscript.\\

\noindent \textbf{Acknowledgments:}
T.S.R. Santos is partially supported by FAPESP grant 2024/15587-1.

\end{document}